\documentclass[10pt]{amsart}
\usepackage{amsmath, amsthm, amsfonts, amssymb}
\usepackage{mathrsfs,color}
\providecommand{\noopsort[1]{}}
\usepackage{bbm}
\usepackage{ifthen}
\numberwithin{equation}{section}
\usepackage{a4}
\usepackage{verbatim}
\allowdisplaybreaks

\newtheorem{thm}{Theorem}[section]
\newtheorem{cor}[thm]{Corollary}
\newtheorem{prop}[thm]{Proposition}
\newtheorem{lem}[thm]{Lemma}

\theoremstyle{remark}
\newtheorem{rem}[thm]{Remark}
\newtheorem{example}[thm]{Example}

\newtheorem{hyp}[thm]{Hypothesis}

\theoremstyle{definition}
\newtheorem{defn}[thm]{Definition}

\newcommand{\eps}{\varepsilon}

\newcommand{\one}{\mathbbm{1}}

\newcommand{\ip}[2]{\ifthenelse{\equal{#1}{}}{\mbox{$ [ \,\cdot\, , \, \cdot \, ] $}}{
\mbox{$ \left[ #1 \, , \, #2 \right]$}}}
\newcommand{\norm}[1]{\ifthenelse{\equal{#1}{}}{\mbox{$\|\cdot\|$}}{\mbox{$\| #1 \|$}}}

\newcommand{\dual}[2]{\ifthenelse{\equal{#1}{}}{\mbox{$ \langle \,\cdot\; , \; \cdot \, \rangle $}}{
\mbox{$ \langle #1   ,  #2 \rangle$}}}

\newcommand{\CR}{\mathbb{R}}

\newcommand{\CN}{\mathbb{N}}

\newcommand{\loc}{\mathrm{loc}}

\newcommand{\A}{\mathscr{A}}

\newcommand{\cL}{\mathscr{L}}

\newcommand{\cB}{\mathscr{B}}

\newcommand{\half}{\frac{1}{2}}

\begin{document}
\title{Kernel estimates for nonautonomous Kolmogorov equations}

\begin{abstract}
Using time dependent Lyapunov functions, we prove pointwise upper bounds for the heat kernels of some  nonautonomous Kolmogorov operators with possibly unbounded drift and diffusion coefficients.
\end{abstract}

\author[M. Kunze]{Markus Kunze}
\address{Graduiertenkolleg 1100, University of Ulm, 89069 Ulm, Germany}
\email{markus.kunze@uni-ulm.de}
\author[L. Lorenzi]{Luca Lorenzi}
\address{Dipartimento di Matematica e Informatica, Universit\`a degli Studi di Parma, Parco Area delle Scienze 53/A, 43124 Parma, Italy}
\email{luca.lorenzi@unipr.it}
\author[A. Rhandi]{Abdelaziz Rhandi}
\address{Dipartimento di Ingegneria dell'Informazione, Ingegneria Elettrica e Matematica Applicata, Universit\`a degli Studi di Salerno, Via Ponte Don Melillo 1, 84084 Fisciano (Sa), Italy}
\email{arhandi@unisa.it}

\maketitle

\section{Introduction}

We study nonautonomous evolution equations
\begin{equation}\label{eq.nee}
\left\{ \begin{array}{rlll}
\partial_t u(t,x) & = & \A (t)u(t,x)\,, & (t,x) \in (s, 1]\times \CR^d\,,\\[1mm]
u(s,x) & = & f(x)\,, & x \in \CR^d\, ,
\end{array}\right.
\end{equation}
where the operators $\A (t)$ are defined on smooth functions $\varphi$ by
\begin{eqnarray*}
(\A(t)\varphi )(x) = \sum_{ij=1}^dq_{ij}(t,x) D_{ij}\varphi (x) +
\sum_{i=1}^dF_i(t,x)D_i\varphi (x)\, ,
\end{eqnarray*}
and $s\in [0,1)$.
Throughout, we make the following assumptions on the coefficients.

\begin{hyp}\label{hyp1}
The coefficients $q_{ij}, F_j$ ($i,j=1,\ldots,d$) are defined on $[0,1]\times \CR^d$ and
\begin{enumerate}
\item
there exists an $\varsigma \in (0,1)$ such that $q_{ij}, F_j \in
C^{\frac{\varsigma}{2}, \varsigma}_\loc ([0,1]\times \CR^d)$ for all
$i,j =1, \ldots, d$. Moreover, $q_{ij} \in C^{0,1}((0,1)\times \CR^d)$;
\item
the matrix $Q = (q_{ij})$ is symmetric and uniformly elliptic in the sense that there exists a number $\eta > 0$ such that
\begin{eqnarray*}
\sum_{i,j=1}^d q_{ij}(t,x)\xi_i\xi_j \geq \eta |\xi|^2 \quad \mbox{for all}\,\, \xi \in \CR^d,\,\, (t,x) \in [0,1]\times \CR^d\,;
\end{eqnarray*}
\item there exist a nonnegative function $V\in C^2(\CR^d)$ and a constant $M \geq 0$ such that $\lim_{|x|\to\infty}V(x) = \infty$
and we have $\A(t)V(x) \leq M$, as well as $\eta \Delta V(x) + F(t,x)\cdot \nabla V(x) \leq M$, for all $(t,x) \in [0,1]\times \CR^d$.
\end{enumerate}
\end{hyp}
Note that neither $q_{ij}$ nor $F_j$ ($i,j =1, \ldots, d$) are assumed to be bounded in $\CR^d$.

Under Hypothesis \ref{hyp1}, it was proved in \cite{kll10} that equation \eqref{eq.nee} is well posed in the sense that, for every
$f \in C_b(\CR^d)$, there exists a unique function $u \in C_b([s, 1]\times \CR^d) \cap C^{1,2}((s, 1]\times \CR^d)$
such that \eqref{eq.nee} is satisfied. Moreover, there exists an evolution family $(G(t,s))_{t,s \in D} \subset \cL (C_b(\CR^d))$, where $D := \{ (t,s) \in [0,1]^2\,:\, t \geq s\}$, such that the unique solution $u$ to \eqref{eq.nee} is given by
$u= G(\cdot,s)f$. It turns out that each operator $G(t,s)$ is a contraction. We recall that an \emph{evolution family} is a family $(G(t,s))_{(t,s) \in D}$ such that $G(t,t) =id_{C_b(\CR^d)}$
and, for $r,s,t \in [0,1]$ with $r \leq s \leq t$, the \emph{evolution law} $G(t,s)G(s,r) = G(t,r)$ holds.
Furthermore, the evolution family can be represented in terms of transition kernels (or transition probabilities) $p_{t,s}$ via the formula
\begin{eqnarray*}
(G(t,s)f)(x) = \int_{\CR^d} f(y)p_{t,s}(x, dy)\,,
\end{eqnarray*}
for each $x\in\CR^d$ and $f \in C_b(\CR^d)$. We refer to Section 2 and \cite{kll10} for a review of these results.

It will be important for us to also consider the adjoint problem to the Cauchy problem \eqref{eq.nee}. Taking adjoints, time is reversed, whence in the
adjoint problem on the measures on $\CR^d$ we have to prescribe \emph{final values} rather than initial values. The adjoint problem is governed by the adjoint operators $G(t,s)^*$. Given $t \in [0,1]$ and a final value $\mu_t$, the solution of the adjoint problem is given by $\mu_s := G(t,s)^*\mu_t$, for $s \in [0,t]$. We stress that the problem of the existence and uniqueness of the solution of the adjoint problem to the Cauchy problem \eqref{eq.nee}, with a prescribed final condition, has been addressed in \cite{bdpr04,bdpr08,bdprs07} even under weaker smoothness assumptions on the coefficients of the operators $\A(t)$ than those we assume in this paper. For more details we refer to the recent survey \cite{bkr09}.

Given an interval $I \subset [0,1]$, we say that a family $(\mu_s)_{s \in I}$ of probability measures
is an \emph{evolution system of measures} if for $t,s \in I$ with $s\leq t$ we have
\begin{equation}\label{eq.esmchar}
G(t,s)^*\mu_t = \mu_s\,.
\end{equation}

Of particular importance is the case where $I=[0,t]$ and $\mu_t = \delta_x$. In this case, $\mu_s = G(t,s)^*\delta_x = p_{t,s}(x, \cdot)$ are exactly the transition probabilities for our evolution equation.

It can be seen that if $t \in (0,1]$ and $(\mu_s)_{s \in [0,t]}$ is an evolution system of measures, then for $s < t$ the measure
$\mu_s$ has a density $\rho (s, \cdot)$ with respect to $d$-dimensional Lebesgue measure.
For more details see Section 2.

The aim of this paper is to study global regularity properties and pointwise bounds of $\rho$ as a function of $(s,y)\in (a,b)\times \CR^d$ for $0<a<b<1$. More precisely, we show, in the case of bounded diffusion coefficients, that $\rho \in W_k^{0,1}((a,b)\times \CR^d)$, provided that \begin{eqnarray*}
\int_{a_0}^{b_0}\int_{\CR^d}|F(\sigma ,x)|^k\rho(\sigma ,x)\,dx\,d\sigma <\infty\,,
\end{eqnarray*}
for all $k>1$, and $0<a_0<a<b<b_0<1$.

Thus, using time dependent Lyapunov functions and an approximation argument we deduce global boundedness and pointwise estimates of $\rho$ in the general case where the diffusion coefficients are not supposed to be bounded.
This is the main result of this paper which generalizes in some sense Theorem 4.1 in \cite{bkr06} and Theorem 3.2 in \cite{brs06}.

In particular, putting $\mu_t = \delta_x$, we  obtain pointwise estimates for the density of the transition probabilities $p_{t,s}(x,\cdot)$ for any $t\in (0,1],\,s\in (0,t)$ and $x\in \CR^d$.

For autonomous and bounded diffusion coefficients similar results can be found in \cite{mpr10,alr10,lmpr11}. For global regularity properties and pointwise estimates in the elliptic and autonomous case, we refer to
\cite{bkr,ffmp09,lr12,mpr05,ms07}.

As an application we show that the transition kernels associated to the operator $(\A (t)\varphi) (x)=(1+|x|^m){\rm Tr}(Q^0(t,x)D^2\varphi(x))-b(t,x)|x|^{p-1}\langle x, \nabla \varphi(x)\rangle $ satisfy
\begin{eqnarray*}
0<p_{t,s}(x,y)\le (t-s)^{-\beta}e^{-\delta_0(t-s)^\alpha |y|^{p+1-m}},\qquad\;\, t\in (0,1],\,s\in (0,t),\,x,y\in \CR^d\,,
\end{eqnarray*}
where $m\ge 0,\,p>\max\{m-1,1\},\,\alpha >(p+1-m)/(p-1)$ and $\delta_0$, $\beta$ are suitable positive constants. Here $Q^0$ and $b$ are, respectively, a matrix valued function and a scalar function satisfying appropriate conditions. This generalizes the examples in \cite{ffmp09,alr10}.

Our approach is a combination of the approaches given in \cite{mpr10,ffmp09}.

\subsection*{Notations}
We write $|x|$ for the Euclidean norm of $x \in \CR^d$ and $\langle x,y\rangle$, or sometimes $x\cdot y$, for the inner product of $x,y \in \CR^d$.
The open ball in $\CR^d$ of radius $r$ and centered at $0$ is denoted by $B_r$. For a matrix $Q \in \CR^{d\times d}$ its trace is denoted by
$\textrm{Tr}(Q)$.

We write $\cB (\CR^d)$ for the Borel $\sigma$-algebra on $\CR^d$, $\delta_x$ for the Dirac measure in $x$ and $\one_E$ for the characteristic
function of a set $E$; $\one=\one_{\CR^d}$.

Concerning derivatives, we use the notations $\partial_tf:=\frac{\partial f}{\partial t}$,
$D_if:=\frac{\partial f}{\partial x_i}$, $D_{ij}f:=\frac{\partial^2f}{\partial x_i\partial x_j}$.
We also write $\nabla f$ for the spatial gradient of $f$. The positive part of a function (or a number) $f$ is denoted by $f_+$.

We assume that the reader is familiar with the spaces $C^r(\CR^d)$ ($r\ge 0$). We add a subscript ``$b$'' (resp.\ ``$c$'')
to denote the subspace of bounded functions (resp.\ functions with compact support).

Now, let $0\leq a < b \leq 1$; we write $Q(a,b)$ for the open cylinder $(a,b)\times \CR^d$ and $\bar{Q}(a,b)$ for its closure $[a,b]\times \CR^d$.
We assume that the reader is familiar with the parabolic Sobolev space $W^{k,l}_p(Q(a,b))$ and with the space $C^{k,l}(Q(a,b))$ ($k,l\in\CN\cup\{0\}$, $p\in [1,\infty]$). If a function $f\in C^{k,l}(Q(a,b))$, along with its derivatives, has continuous extensions
to $\bar{Q}(a,b)$, it belongs to $C^{k,l}(\bar{Q}(a,b))$. We also consider the spaces
\begin{align*}
C_c^{k,l}(Q(a,b)) &= \{\phi \in C^{k,l}(Q(a,b))\,:\, \mathrm{supp}(\phi) \subset (a,b)\times B_R \; \mbox{for some } \, R>0\}\,,\\
C_c^{k,l}(\bar{Q}(a,b)) &= \{\phi \in C^{k,l}(\bar{Q}(a,b))\,:\, \mathrm{supp}(\phi) \subset [a,b]\times B_R \; \mbox{for some } \, R>0\}\,.
\end{align*}
Note that we are not requiring that $u\in C_c^{k,l}(\bar{Q}(a,b))$ vanishes at
$t=a,\,t=b$.
The  canonical norm in $L^p(a,b; L^q(\CR^d))$ is denoted by $\|\cdot\|_{p,q}$. When $p=q$, we either write
$\|\cdot\|_{L^p(Q(a,b))}$ or simply $\|\cdot\|_p$.

For $\varsigma\in (0,1)$, $C^{\varsigma/2,\varsigma}(\bar{Q}(a,b))$
is the usual parabolic H\"older space. We use the subscript ``loc'' to denote
the space of all $f\in C(\bar{Q}(a,b))$ which are $(\varsigma/2,\varsigma)$-H\"older continuous in  $[a,b]\times K$ for any
compact set $K\subset\CR^d$. Similarly, $C^{1+\varsigma/2,2+\varsigma}(\bar{Q}(a,b))$ (resp. $C^{1+\varsigma/2,2+\varsigma}_{\rm loc}(\bar{Q}(a,b))$)
is the set of all functions $f$ such that $\partial_t f$, $D_if$ and $D_{ij}f$ ($i,j=1,\ldots,d)$ belong to $C^{\varsigma/2,\varsigma}(\bar{Q}(a,b))$ (resp. $C^{\varsigma/2,\varsigma}_{\rm loc}(\bar{Q}(a,b))$).

We use also the space $\mathscr{H}^{p,1}(Q(a,b))$ (cf. \cite{krylov01}) of all functions $u\in W^{0,1}_p(Q(a,b))$
such that the distributional derivative $\partial_tu$ belongs to the space $(W^{0,1}_{p'}(Q(a,b)))'$, where $1/p+1/p'=1$.
The space $\mathscr{H}^{p,1}(Q(a,b))$ is endowed with the norm
\begin{eqnarray*}
\|f\|_{{\mathscr H}^{p,1}(Q(a,b))}=\|\partial_t f\|_{(W^{0,1}_{p'}(Q(a,b)))'}+\|f\|_{W^{0,1}_p(Q(a,b))}\,.
\end{eqnarray*}
With a slight abuse of notation, we indicate by
\begin{eqnarray*}
\int_{Q(a,b)}v\partial_tu\,dt\,dx
\end{eqnarray*}
the pairing between $\partial_t u\in (W^{0,1}_{p'}(Q(a,b)))'$ and $v\in W^{0,1}_{p'}(Q(a,b))$.

\section{Preliminaries}

\subsection{The associated evolution family}

We recall some basic properties of the evolution
family $(G(t,s))_{(t,s) \in D}$.

\begin{lem}\label{l.Gproperties}
For $(t,s) \in D$ and $x \in \CR^d$, there exists a unique Borel probability measure $p_{t,s}(x, \cdot)$ such that
\begin{eqnarray*}
(G(t,s)f)(x) = \int_{\CR^d} f(y)p_{t,s}(x, dy)\,,
\end{eqnarray*}
for each $f \in C_b(\CR^d)$. Moreover, the following properties hold.
\begin{enumerate}
\item $p_{t,t}(x, \cdot) =\delta_x$. For $s<t$, the measure $p_{t,s}(x, \cdot)$ is equivalent
to the Lebesgue measure.
\item For fixed $A \in \cB (\CR^d)$ and $(t,s) \in D$, the map $x \mapsto p_{t,s}(x,A)$ is Borel measurable.
\item If $A \in \cB(\CR^d)$ has strictly positive Lebesgue measure, then $p_{t,s}(x,A) = (G(t,s)\one_A)(x) > 0$, for
all $(t,s) \in D$, with $t>s$, and $x\in\CR^d$.
\end{enumerate}
\end{lem}

\begin{proof}
(1) and (2) are part of \cite[Proposition 2.4]{kll10}, (3) is \cite[Corollary 2.5]{kll10}.
\end{proof}

In view of Lemma \ref{l.Gproperties}, $p_{t,s}(x,dy)$ admits a density with respect to the Lebesgue measure for any $0<s<t\le 1$ and any $x\in\CR^d$, which with a slight abuse of notations we also denote by $p_{t,s}(x,y)$.

By construction of the evolution operator, it is clear that for
$f \in C_b(\CR^d)$, $x \in \CR^d$ and $s \in [0,1)$, the map $ t \mapsto (G(t,s)f)(x)$ is differentiable and $\frac{d}{dt}(G(t,s)f)(x) =
(\A(t)G(t,s)f)(x)$ for $t> s$. In what follows, it will be important to have also information about the derivative with respect to $s$.

\begin{lem}\label{l.sderivative}
Let $\varphi \in C^{1,2}_c(\bar{Q}(0,1))$. Then, for $t \in (0,1]$ and $x \in \CR^d$, the function $s \mapsto (G(t,s)\varphi (s))(x)$
is differentiable in $[0, t]$ and
\begin{eqnarray*}
\partial_s(G(t,s)\varphi(s)) = G(t,s)\partial_s\varphi(s) - G(t,s)\A (s)\varphi (s)\quad \mbox{for all}\, \,s \in [0,t]\,.
\end{eqnarray*}
\end{lem}

\begin{proof}
In the case where $\varphi$ is replaced with a function $\psi \in C_c^2(\CR^d)$, the result was proved in
\cite[Lemma 3.2]{kll10}.
To prove the general case, fix $t\in (0,1]$, $s \in [0,t]$ and $x\in\CR^d$. Then, for $h \in \CR$ such that $s+h \in [0,t]$ we have
\begin{equation}\label{eq.differencequotient}
\begin{aligned}
& \frac{(G(t,s+h)\varphi (s+h))(x)- (G(t,s)\varphi (s))(x)}{h}\\
 = &\bigg [G(t,s+h)\frac{\varphi (s+h)- \varphi (s)}{h}\bigg ](x)+
\frac{(G(t,s+h)\varphi (s))(x) - (G(t,s)\varphi (s))(x)}{h}\,.
\end{aligned}
\end{equation}
Applying \cite[Lemma 3.2]{kll10} with $\psi (x) := \varphi (s,x)$, it follows that the second term on the right hand side of \eqref{eq.differencequotient}
converges to $-(G(t,s)\A(s)\varphi (s))(x)$ as $h \to 0$. To handle the first term, let us write $\Delta_h(y) := h^{-1}(\varphi (s+h, y) - \varphi (s,y))$.
Since $\partial_s\varphi$ is uniformly continuous, it follows that $\Delta_h \to \partial_s\varphi(s)$, uniformly on $\CR^d$.
Using that $\|G(t,s+h)\|_{\cL (C_b(\CR^d))} = 1$, we see that
\begin{align*}
&|(G(t,s+h)\Delta_h)(x) - (G(t,s)\partial_s\varphi(s))(x)|\\
\leq & \|\Delta_h - \partial_s\varphi(s)\|_\infty + |(G(t,s+h)\partial_s\varphi(s))(x) - (G(t,s)\partial_s\varphi(s))(x)| \to 0\,,
\end{align*}
as $h \to 0$, since the last term in the equation above tends to $0$ as $h \to 0$ as a consequence of \cite[Lemma 3.2]{kll10}.
\end{proof}

Let us note some properties of evolution systems of measures.

\begin{lem}\label{l.simproperties}
Let $t \in (0,1]$ and $(\mu_s)_{s\in [0,t]}$ be an evolution system of measures. Then
\begin{enumerate}
\item
for $s<t$, the measure $\mu_s$ is absolutely continuous with respect to Lebesgue measure;
\item
for $A \in \cB(\CR^d)$, the map $[0,t] \ni s \mapsto \mu_s(A)$ is Borel measurable.
\end{enumerate}
\end{lem}

\begin{proof}
(1) It follows from the structure of $G(t,s)$ that $\mu_s(A) = \big[G(t,s)^*\mu_t\big](A)=
\int_{\CR^d} p_{t,s}(x,A)\,d\mu_{t}(x)$. If $A \in \cB(\CR^d)$ has
Lebesgue measure zero, then $p_{t,s}(x,A) \equiv 0$ for $s<t$ by Lemma \ref{l.Gproperties}(1). Thus $\mu_s(A) = 0$.

(2) By \cite[Lemma 3.2]{kll10}, for $f \in C^2_c(\CR^d)$ the function $s \mapsto (G(t,s)f)(x)$ is continuous. Thus, for such a function $f$ and a probability measure $\nu$
also the function
\begin{eqnarray*}
s \mapsto \int_{\CR^d} G(t,s)f\, d\nu
\end{eqnarray*}
is continuous, which follows easily from dominated convergence. Now, let $B$ be an open ball in $\CR^d$. Then, there exists a uniformly bounded sequence $f_n$ of $C^2_c$-functions
such that $f_n \downarrow \one_B$. By dominated convergence,
\begin{eqnarray*}
\mu_s(B) = \int_{\CR^d}G(t,s)\one_B\, d\mu_t = \lim_{n\to \infty}\int_{\CR^d} G(t,s)f_n\, d\mu_t
\end{eqnarray*}
is measurable as the pointwise limit of continuous functions. A monotone class argument shows that $s \mapsto \mu_s(A)$ is measurable for all $A \in \cB(\CR^d)$.
\end{proof}

\begin{lem}\label{l.differential}
Fix $t\in (0,1]$ and let $(\mu_s)_{s \in [0,t]}$ be an evolution system of measures, $\varphi \in C^{1,2}_c(\bar{Q}(0,t))$
and $0\leq a<b\leq t$. Then
\begin{eqnarray*}
\int_a^b \int_{\CR^d}\big [\partial_s \varphi  - \A (s)\varphi\big ] \, d\mu_s\, ds =
\int_{\CR^d} \varphi (b)\, d\mu_b - \int_{\CR^d}\varphi (a)\, d\mu_a\,.
\end{eqnarray*}
\end{lem}

\begin{proof}
Using Lemma \ref{l.sderivative} and Equation \eqref{eq.esmchar}, we find
\begin{align*}
\int_{\CR^d}\big [\partial_s\varphi (s) - \A (s)\varphi (s)\big ] \, d\mu_s
= & \int_{\CR^d}G(b,s)\Big[ \partial_s \varphi (s) - \A (s)\varphi (s) \Big]\, d\mu_b\\[1mm]
= & \int_{\CR^d}\partial_s (G(b,s)\varphi (s))\, d\mu_b\, .
\end{align*}
Integrating this equality from $a$ to $b$ and using \eqref{eq.esmchar} again, the claim follows.
\end{proof}

Given an evolution system of measures $(\mu_s)_{s \in [0,t]}$, it is a consequence of Lemma \ref{l.simproperties} that
we can define
\begin{equation}\label{eq.nu}
\nu (A\times B) := \int_A \mu_s(B)\, ds\,,
\end{equation}
for $A \in \cB ((0,t))$ and $B \in \cB (\CR^d)$. We may then extend $\nu$ in a straightforward way to a Borel measure on
$(0,t)\times \CR^d$.
\begin{rem}
Lemma \ref{l.differential} yields in particular that for $\varphi  \in C_c^{1,2}(Q(a,b))$ we have
\begin{eqnarray*}
\int_{Q(a,b)} (\partial_s - \A)\varphi \, d\nu = 0\, ,
\end{eqnarray*}
so that $(\partial_s - \A)^*\nu = 0$ in the sense of \cite{bkr06}. Thus the local regularity results obtained in that reference are available, which show, in particular, that $\nu$ is absolutely continuous with respect to the
Lebesgue measure on $(0,t)\times \CR^d$ and its density $\rho: (0,t)\times \CR^d \to [0, \infty)$ (with respect to
the Lebesgue measure on $(0,t)\times\CR^d$) is locally $\gamma$-H\"older continuous in $(0,t)\times\CR^d$ of any exponent $\gamma\in (0,1)$.
\end{rem}

\subsection{Time dependent Lyapunov functions}
Similar as in \cite{s08,alr10} we use time dependent Lyapunov functions to prove kernel estimates. We will use the time-independent
Lyapunov function $V$ from Hypothesis \ref{hyp1} to ensure certain integrability properties of the time dependent Lyapunov functions.

\begin{defn}
\label{defn-2.6}
Let $t \in (0,1]$. A \emph{time dependent Lyapunov function $($on $[0,t])$} is a function
$0 \leq W \in C(\bar{Q}(0,t))\cap C^{1,2}(Q(0,t))$ such that
\begin{enumerate}
\item
$W(s,x) \leq V(x)$ for all $(s,x) \in [0,t]\times \CR^d$;
\item
$\lim_{|x|\to \infty}W(s,x) = \infty$, uniformly for $s$ in compact subsets of $[0,t)$;
\item there exists a function $0\leq h \in L^1((0,t), ds)$ such that
\begin{eqnarray*}
\partial_s W(s,x) -(\A(s)W(s))(x) \geq -h(s)W(s,x)
\end{eqnarray*}
and
\begin{eqnarray*}
\partial_sW(s,x) - (\eta \Delta W(s,x)+F(s,x)\cdot \nabla W(s,x)) \geq - h(s)W(s,x)\,,
\end{eqnarray*}
for all $(s,x) \in (0,t)\times \CR^d$.
\end{enumerate}
To stress the dependence on $h$, we will sometimes say that $W$ is a time dependent Lyapunov function \emph{with respect to $h$}.
\end{defn}

We now prove the following extension of \cite[Proposition 2.3]{s08} to the nonautonomous setting.

\begin{prop}\label{p.lyapunov}
Let $t \in (0,1]$ and $(\mu_s)_{s \in [0,t]}$ be an evolution system of measures such that $V \in L^1(\mu_t)$. Moreover, let $W$ be a time dependent
Lyapunov function on $[0,t]$ $($with respect to $h)$. Then, $W(s) \in L^1(\mu_s)$ for $s \in [0,t]$ and we have
\begin{equation}\label{eq.lyapunovest}
\int_{\CR^d} W(s)\, d\mu_s \leq e^{\int_s^t h(\tau)\, d\tau} \int_{\CR^d} W(t)\, d\mu_t\,.
\end{equation}
\end{prop}

\begin{proof}
Let us first note that
\begin{equation}\label{eq.vest}
\int_{\CR^d}V\, d\mu_s \leq \int_{\CR^d} V\, d\mu_t+ M(t-s)\, ,
\end{equation}
where $M$ is defined in Hypothesis \ref{hyp1}(3).
Indeed, it was shown in the proof of \cite[Lemma 3.4]{kll10} that  for $s<t$ and $x\in\CR^d$ we have
\begin{eqnarray*}
(G(t,s)V)(x) \leq V(x) + M (t-s)\, ,
\end{eqnarray*}
which is precisely \eqref{eq.vest} in the case where $\mu_t = \delta_x$. The general case follows by integrating the above inequality with respect
to the given probability measure $\mu_t$ and using the formula $G(t,s)^*\mu_t =\mu_s$.
Since $W(s) \leq V$ for any $s\in [0,t]$, it follows that $W(s) \in L^1(\mu_s)$ and $\int_{\CR^d} W(s)\, d\mu_s \leq \int_{\CR^d} V\, d\mu_t + M =:\tilde{M}$.\medskip

Now, fix $0\leq s<r < t$ and pick a sequence of functions $\psi_n \in C^{\infty}([0,\infty))$ such that
\begin{itemize}
\item[(i)] $\psi_n(\tau) = \tau$ for $\tau \in [0,n]$;
\item[(ii)] $\psi_n(\tau) \equiv$ const. for $\tau\geq n+1$;
\item[(iii)] $0\leq \psi_n'\leq 1$ and $\psi_n''\leq 0$.
\end{itemize}
Since $W(\sigma,x) \to \infty$ as $|x| \to \infty$ uniformly on $[0,r]$, the restriction to $\bar Q(0,r)$ of the function
$W_n:= \psi_n\circ W$ can be written as a function in $C_c^{1,2}(Q(0,r))$ plus a constant function. As the assertion of Lemma \ref{l.differential} is obviously true for constants, we can apply it to $W_n$ and obtain
\begin{align}
&\int_{\CR^d} W_n(r) \, d\mu_r - \int_{\CR^d} W_n(s)\, d\mu_s\notag\\
=&\int_s^r\int_{\CR^d}\big [\partial_{\sigma}W_n(\sigma)-\A(\sigma)W_n(\sigma)\big ]\,d\mu_{\sigma}\,d\sigma\notag\\
= & \int_s^r \int_{\CR^d} \big [\psi_n'(W(\sigma))\big (\partial_\sigma W(\sigma) - \A (\sigma)W(\sigma)\big)\big ]\, d\mu_{\sigma}\, d\sigma\notag\\
& - \int_s^r \int_{\CR^d} \big [\psi_n''(W(\sigma))\big(Q(\sigma)\nabla W(\sigma)\cdot \nabla W(\sigma)\big)\big ]
 \, d\mu_{\sigma}\, d\sigma\notag\\
\geq & -\int_s^r \int_{\CR^d}\psi_n'(W(\sigma))h(\sigma)W(\sigma)\, d\mu_\sigma\, d\sigma\, .
\label{eq.last}
\end{align}

Let us fix an increasing sequence $(r_k)\in [0,t)$ converging to $t$ as $k\to\infty$. From \eqref{eq.last} with $r=r_k$, we get
\begin{align*}
\int_{\CR^d} W_n(r_k) \, d\mu_{r_k} - \int_{\CR^d} W_n(s)\, d\mu_s
\ge  -\int_s^{r_k} \int_{\CR^d}\psi_n'(W(\sigma))h(\sigma)W(\sigma)\, d\mu_\sigma\, d\sigma\, .
\end{align*}
We now want to let $k$ tend to $\infty$. Clearly, we have only to discuss the convergence of the first term in the left-hand side of
the previous inequality. Note that
\begin{align}
& \int_{\CR^d} W_n(r_k) \, d\mu_{r_k}
-\int_{\CR^d}W_n(t) \, d\mu_t\notag\\
=&\int_{\CR^d}\big [W_n(r_k)-W_n(t)\big ] \, d\mu_{r_k}+\int_{\CR^d}W_n(t)\,d\mu_{r_k}-\int_{\CR^d}W_n(t)\,d\mu_{t}\notag\\
=&\int_{\CR^d}\big [W_n(r_k)-W_n(t)\big ] \, d\mu_{r_k}+\int_{\CR^d}\big [G(t,r_k)W_n(t)-W_n(t)\big ]\,d\mu_{t}\,.
\label{eq.last-0}
\end{align}
By \cite[Theorem 3.7]{kll10} the function $G(t, \cdot)f$ is continuous in $[0,t]\times\CR^d$ for any $f\in C_b(\CR^d)$.
This shows first that the second term in the last side of \eqref{eq.last-0} tends to $0$ as $k\to\infty$.
Moreover, it follows that the map $s \mapsto G(t,s)^*\mu_t$ is weakly continuous. Thus the set of measures $\{\mu_s: s \in [0,t]\}$ is weakly compact and hence, by Prokhorov's Theorem, tight.
Consequently, given $\varepsilon>0$, there exists $m>0$ such that $\mu_{r_k}(\CR^d\setminus B_m)\le \varepsilon$ for every $k\in\CN$.
We can thus estimate
\begin{align*}
&\left |\int_{\CR^d}\big [W_n(r_k)-W_n(t)\big ] \, d\mu_{r_k}\right |\\
\le &\int_{B_m}|W_n(r_k)-W_n(t)|\,d\mu_{r_k}
+\int_{\CR^d\setminus B_m}|W_n(r_k)-W_n(t)|\,d\mu_{r_k}\\
\le &\sup_{x\in B_m}|W_n(r_k,x)-W_n(t,x)|
+\|W_n(r_k)-W_n(t)\|_{\infty}\mu_{r_k}(\CR^d\setminus B_m)\\
\le &\sup_{x\in B_m}|W_n(r_k,x)-W_n(t,x)|
+2\|W_n\|_{\infty}\eps\,,
\end{align*}
for any $k\in\CN$. Since $W$ is continuous and $\eps>0$ is arbitrary, it follows that
\begin{align*}
\lim_{k\to\infty}\int_{\CR^d}\big [W_n(r_k)-W_n(t)\big ] \, d\mu_{r_k}=0\,.
\end{align*}
Summing up, we have proved that
\begin{eqnarray*}
\lim_{k\to\infty}\int_{\CR^d} W_n(r_k) \, d\mu_{r_k}
=\int_{\CR^d}W_n(t) \, d\mu_t
\end{eqnarray*}
and, consequently,
\begin{align}
\int_{\CR^d} W_n(t) \, d\mu_t - \int_{\CR^d} W_n(s)\, d\mu_s
\ge  -\int_s^t \int_{\CR^d}\psi_n'(W(\sigma))h(\sigma)W(\sigma)\, d\mu_\sigma\, d\sigma\, .
\label{eq.last-1}
\end{align}

Since
\begin{eqnarray*}
\big|\psi_n'(W(\sigma,x)) h(\sigma)W(\sigma,x)| \leq h(\sigma) V(x)
\end{eqnarray*}
for any $\sigma\in [s,t]$ and any $x\in\CR^d$,
and
\begin{eqnarray*}
\int_s^t\int_{\CR^d} h(\sigma)V\, d\mu_\sigma \, d\sigma \leq \tilde{M}\int_s^t h(\sigma)\, d\sigma < \infty\,,
\end{eqnarray*}
we can use dominated convergence in \eqref{eq.last-1} and obtain upon $n\to\infty$
\begin{eqnarray*}
\int_{\CR^d} W(t)\, d\mu_t -\int_{\CR^d}W(s)\, d\mu_s \geq  - \int_s^t h(\sigma) \int_{\CR^d} W(\sigma)\, d\mu_\sigma\, d\sigma\, .
\end{eqnarray*}
Writing $\zeta (s) := \int_{\CR^d} W(s)\, d\mu_s$ we have thus proved that
\begin{eqnarray}\label{eq2-7'}
\zeta (t) - \zeta (s) \geq -\int_s^t h(\tau) \zeta (\tau) \, d\tau\, .
\end{eqnarray}
We claim that this implies \eqref{eq.lyapunovest}. Indeed, the function $\Phi$, defined by
\begin{eqnarray*}
\Phi (\tau) := \Big( \zeta(t) + \int_\tau^t h(\sigma) \zeta (\sigma)\, d\sigma \Big) e^{\int_s^\tau h(\sigma)\, d\sigma}\,,
\end{eqnarray*}
is continuous on $[s,t]$ and therein weakly differentiable with
\begin{eqnarray*}
\Phi'(\tau) =  h(\tau) \bigg( \zeta (t) + \int_\tau^t h(\sigma)\zeta (\sigma)\, d\sigma -\zeta(\tau)\bigg)e^{\int_s^\tau h(\sigma)\, d\sigma}
\geq 0\, ,
\end{eqnarray*}
by (\ref{eq2-7'}). Thus, $\Phi$ is increasing
and, using again \eqref{eq2-7'}, we find
\begin{eqnarray*}
\zeta(s) \leq \Phi (s) \leq \Phi (t) = \zeta (t)e^{\int_s^t h (\sigma)\, d\sigma} \, .
\end{eqnarray*}
The proof is now complete.
\end{proof}

\subsection{Approximation of the coefficients}\label{subsect.approx}

In the proof of our main result we will approximate the diffusion coefficients $q_{ij}$ ($i,j=1,\ldots,d$)
with bounded diffusion coefficients $q_{ij}^{(n)}$ so that we can apply Theorem \ref{t.auxiliary} below, which assumes bounded diffusion coefficients.
Let us now describe how we construct the coefficients $q_{ij}^{(n)}$.

Suppose that we are given $t \in (0,1]$ and a time dependent Lyapunov function $W$ on $[0,t]$.
We pick a function $\varphi \in C_c^\infty (\CR)$ such that $\one_{(-1,1)} \leq \varphi \leq \one_{(-2,2)}$ and such that
$|t\varphi'(t)| \leq 2$ for all $t \in \CR$.
We then define
\begin{eqnarray*}
\varphi_n (s,x) := \varphi \bigg( \frac{W(s,x)}{n}\bigg)\,,
\end{eqnarray*}
for any $(s,x)\in [0,t]\times\CR^d$ and any $n\in\CN$,
and set
\begin{equation}\label{eq.qn}
q_{ij}^{(n)}(s,x) := \varphi_n (s,x) q_{ij}(s,x) + (1-\varphi_n(s,x))\eta \delta_{ij}\,,
\end{equation}
for any $i,j=1,\ldots,d$, where $\delta_{ij}$ is the Kronecker delta. Finally, we define
\begin{equation}\label{eq.an}
\A_n(s) := \sum_{i,j=1}^d q_{ij}^{(n)}(s)D_{ij} + \sum_{j=1}^d F_j(s) D_j\, .
\end{equation}

\begin{lem}\label{l.approx}
For every $n \in \CN$ the operator $\A_n$ satisfies Hypothesis \ref{hyp1} with the interval $[0,1]$ replaced with $[0,t]$.
Moreover, the coefficients $q_{ij}^{(n)}$, along with their first order spatial derivatives are bounded on $Q(0,t)$. Finally,
when $\tilde{W}$ is a time dependent Lyapunov function on $[0,t]$ for the operator $\partial_s -\A$, then it is a time dependent Lyapunov function for the operator $\partial_s-\A_n$ with respect to the same $h$.
\end{lem}

\begin{proof}
Clearly, the coefficients $q_{ij}^{(n)}$ are locally H\"older continuous functions and $Q_n=(q_{ij}^{(n)})$ is uniformly elliptic with the same constant $\eta$. Since the functions $\varphi_n$ vanish outside a compact set, the coefficients $q_{ij}^{(n)}$ and their spatial derivatives are bounded, continuous functions on $[0,t]\times \CR^d$.
Moreover,
\begin{eqnarray*}
\A_n(s) V  =  \varphi_n \A(s) V + (1-\varphi_n)(\eta \Delta V + F(s)\cdot \nabla V) \leq \varphi_n M + (1-\varphi_n)M = M\,,
\end{eqnarray*}
for all $(s,x) \in [0,t]\times\CR^d$. The assertion concerning $\tilde{W}$ is proved similarly.
\end{proof}

As a consequence of Lemma \ref{l.approx}, for every $n \in \CN$ we have an evolution family $(G_n(r,s))_{0\leq s\leq r \leq t}$
associated with equation \eqref{eq.nee} where $\A$ is replaced with $\A_n$. Thus, given a measure $\mu_t^n$, we can define an evolution system of measures $(\mu_s^n)_{s \in [0,t]}$ by setting $\mu_s^n := G_n(t,s)^*\mu_t^n$. Taking Lemma \ref{l.simproperties} into account, we can also define the measure $\nu_n$ on $(0,t)\times \CR^d$ by setting
\begin{eqnarray*}
\nu_n(A\times B) = \int_A\mu_s^n(B)\, ds\,,
\end{eqnarray*}
for any $A\in\cB((0,t))$, any $B\in \cB(\CR^d)$, and then extending $\mu_n$ in the standard way to $\cB((0,t)\times\CR^d)$.
We denote the density of $\nu_n$ with respect to the Lebesgue measure by $\rho_n$.
As before, $\mu_s$, $\nu$ and $\rho$ refer to the corresponding objects associated with $\A$.

\begin{prop}\label{p.conv}
If $\mu_t^{n}$ converges weakly to $\mu_t$, then $\rho_n \to \rho$ locally uniformly in $Q(0,t)$.
\end{prop}

\begin{proof}
For $f \in C_c^{2+\varsigma}(\CR^d)$ the unique solution in $C^{1,2}((s,1)\times\CR^d)\cap C_b([s,1]\times\CR^d)$ of the equation $\partial_tu(t) = \A_n(t) u(t)$ with $u(s) =f$ is given by
$u_n = G_n(\cdot,s)f$. Combining interior Schauder estimates with a diagonal argument, we see that, up to a subsequence, the functions
$u_n$, together with their first order time derivative and their first and second order spatial derivatives, converge locally uniformly to
a function $u$ (resp.\ its derivatives).
Clearly, $u$ must solve the equation $\partial_tu(t) = \A(t)u(t)$ with $u(s) = f$, whence $u = G(\cdot,s)f$.\smallskip

We claim that this implies that $\mu_s^n$ converges weakly to $\mu_s$ for all $s \in [0,t)$. Indeed, for $f \in C_c^{2+\varsigma}(\CR^d)$ we have
\begin{align*}
&\bigg| \int_{\CR^d} f \, d\mu_s^n - \int_{\CR^d} f\, d\mu_s\bigg|\\
= & \bigg| \int_{\CR^d} G_n(t,s)f\, d\mu_t^n - \int_{\CR^d} G(t,s)f\, d\mu_t \bigg|\\
\leq & \int_{\CR^d} |G_n(t,s)f - G(t,s)f|\, d\mu_t^n + \bigg| \int_{\CR^d} G(t,s)f\, d\mu_t^n - \int_{\CR^d} G(t,s)f\, d\mu_t \bigg| \, ,
\end{align*}
and the latter converges to zero. For the last term, this follows from the weak convergence of $\mu_t^n$ to $\mu_t$.
For the first summand, one uses the tightness of the measures $\mu_t^n$ (which follows from the
 weak convergence of $\mu_t^n$ to $\mu_t$ and Prokhorov's theorem) and the uniform convergence on compact sets of $G_n(t,s)f$ to $G(t,s)f$, proved above.

This proves that $\int_{\CR^d} f \, d\mu_s^{n} \to \int_{\CR^d} f \, d\mu_s$ for all $f \in C_c^{2+\varsigma}(\CR^d)$. The latter set is convergence determining, whence $\mu_s^n$ converges weakly to $\mu_s$.
Since this is true for every $s$, it follows that $\nu_n$ converges weakly to $\nu$. Indeed, if
$f \in C_b(Q(0,t))$, then
\begin{eqnarray*}
\int_{\CR^d} f \, d\nu_n = \int_0^t\int_{\CR^d} f(s,x)\, d\mu_s^n(x)\, ds \to \int_0^t\int_{\CR^d} f(s,x)\, d\mu_s(x)\, ds\,
\end{eqnarray*}
by dominated convergence.\smallskip

It follows from \cite[Corollary 3.11]{bkr06} and Sobolev embedding that, for any compact set $K \subset \CR^d$ and any compact interval
$J \subset (0,t)$, there exist a constant $C>0$ and $\gamma\in (0,1)$, independent of $n$, such that $\|\rho_n\|_{C^\gamma (J\times K)} \leq C$. Thus, by compactness and a diagonal argument, a subsequence of $\rho_n$ converges locally uniformly to a continuous function $\psi$. Since $\nu_n$ converges
weakly to $\nu$, it is easy to see that we must have $\psi = \rho$. A subsequence-subsequence argument shows that $\rho_n$ converges
locally uniformly to $\rho$.
\end{proof}

\section{Results for bounded diffusion coefficients}

We now proceed to prove some results in the case where the diffusion coefficients $q_{ij}$ are additionally bounded. These results will later on be applied
to the parabolic equation associated with the approximate coefficients $q_{ij}^{(n)}$, constructed in Section \ref{subsect.approx}. However, for ease
of notation, we suppress the index $n$ here.

\subsection{Global regularity results} We first address the question for which values of $r$ the density $\rho$ belongs to
the space $L^r(Q(a,b))$ and/or the space $\mathscr{H}^{r,1}(Q(a,b))$. We follow the strategy in \cite{mpr10}, where the autonomous situation was considered. The proofs are straightforward generalizations of the ones given there. However, in an effort of being self contained, we provide full proofs.
\smallskip

Throughout this subsection, we assume that the coefficients $q_{ij}$ and their spatial derivatives $D_kq_{ij}$ ($i,j,k =1, \ldots, d$) are bounded on
$(0,1)\times \CR^d$.

For $0\leq a < b \leq 1$ and $k \geq 1$, we define $\Gamma (k,a,b)$ by
\begin{eqnarray*}
\Gamma (k,a,b) := \bigg ( \int_a^b \int_{\CR^d} |F(\sigma ,x)|^k \rho(\sigma ,x)\, dx\,d\sigma \bigg)^\frac{1}{k}\, .
\end{eqnarray*}

\begin{prop}\label{p.qinlr}
Let $0<a<b<1$.
If $\Gamma (1,a,b) < \infty$, then $\rho \in L^r(Q(a,b))$ for all $r \in [1, (d+2)/(d+1))$ and
\begin{eqnarray*}
\|\rho\|_{L^r(Q(a,b))} \leq C(1+\Gamma (1,a,b))\,,
\end{eqnarray*}
for some constant $C$, depending only on the coefficients $q_{ij}$ $(i,j=1,\ldots,d)$ and the dimension $d$.
\end{prop}

\begin{proof}
Let us write $\A_0:= \sum_{i,j=1}^d q_{ij}D_{ij}$. By Lemma \ref{l.differential} we have
\begin{eqnarray*}
\int_{Q(a,b)} (\partial_s\varphi - \A_0\varphi)\rho\, ds\,dx = \int_{\CR^d} \varphi(b)\, d\mu_b - \int_{\CR^d} \varphi (a)\, d\mu_a
+ \int_{Q(a,b)}(F\cdot \nabla \varphi)\rho\, ds\,dx\,,
\end{eqnarray*}
for all $\varphi \in C_c^{1,2}(\bar{Q}(0,1))$. Since $\mu_t$ is a probability measure
for all $t \in (0,1)$, we infer that
\begin{align}
\bigg| \int_{Q(a,b)}  (\partial_s\varphi - \A_0\varphi)\rho\, ds\, dx \bigg| & \leq  \Gamma (1,a,b) \|\varphi\|_{W^{0,1}_\infty (Q(a,b))}
+ 2 \|\varphi\|_\infty\notag\\
&\leq  (2 + \Gamma (1,a,b))\|\varphi\|_{W^{0,1}_\infty (Q(a,b))}\,.
\label{eq.gammaest}
\end{align}
Now, let $\psi \in C_c^\infty(Q(a,b))$. By Schauder theory, see \cite[Theorem 9.2.3]{krylov}, the parabolic problem
\begin{eqnarray*}
\left\{ \begin{array}{rlll}
\partial_t u(t,x) - (\A_0(t)u(t))(x) & = & \psi(t,x)\,,& (t,x) \in Q(0,1)\,,\\[1mm]
u(0,x) & = & 0, & x \in \CR^d\,,
\end{array}
\right.
\end{eqnarray*}
has a unique solution $u \in C^{1+\frac{\varsigma}{2}, 2+\varsigma}(\bar{Q}(0,1))$. Pick $r' > d+2$. It follows from
\cite[Theorem IV.9.1]{lsu}, that the solution $u$ belongs to $W_{r'}^{1,2}(Q(0,1))$ and we have
\begin{equation}\label{eq.est1}
\|u\|_{W_{r'}^{1,2}(Q(0,1))} \leq C \|\psi\|_{L^{r'}(Q(a,b))}\, ,
\end{equation}
for some positive constant $C$ independent of $\psi$.
As $r' > d+2$, the space $W_{r'}^{1,2}(Q(0,1))$ is continuously embedded into $W_\infty^{0,1}(Q(0,1))$, see
\cite[Lemma II.3.3]{lsu}. Combining this with \eqref{eq.est1}, we obtain
\begin{equation}\label{eq.est2}
\|u\|_{W_\infty^{0,1}(Q(0,1))} \leq C' \|\psi\|_{L^{r'}(Q(a,b))}\, .
\end{equation}
Now let $\vartheta \in C_c^\infty (\CR^d)$ be such that $\one_{B_1}\leq \vartheta \leq \one_{B_2}$ and define, for $n\in\CN$, the function
$\varphi_n$ by
$\varphi_n (t,x) := \vartheta (x/n)u(t,x)$. Then, $\varphi_n$ satisfies the assumption of Lemma \ref{l.differential}, hence
\eqref{eq.gammaest} is valid for $\varphi$ replaced with $\varphi_n$. Combining with \eqref{eq.est2}, we obtain
\begin{align*}
\bigg| \int_{Q(a,b)} (\partial_s\varphi_n - \A_0\varphi_n)\rho\, ds\, dx \bigg|
&\leq (2+\Gamma (1,a,b))\|\varphi_n\|_{W^{0,1}_\infty (Q(a,b))}\\
&\leq  C''(1+\Gamma (1,a,b))\|u\|_{W^{0,1}_\infty (Q(a,b))}\\
&\leq C''C'(1+\Gamma(1,a,b) )\|\psi\|_{L^{r'}(Q(a,b))}\,,
\end{align*}
for any $n\in\CN$.
Letting $n\to \infty$, we obtain
\begin{eqnarray*}
\bigg| \int_{Q(a,b)} \psi \rho\, ds\, dx \bigg| \lesssim (1+\Gamma(1,a,b) )\|\psi\|_{L^{r'}(Q(a,b))}\, .
\end{eqnarray*}
The arbitrariness of $\psi \in C_c^{\infty}(Q(a,b))$ implies that
$\rho \in L^r(Q(a,b))$ where $1/r+1/r' = 1$. Since $r'> d+2$ was arbitrary, this is true for any $r \in [1, (d+2)/(d+1))$.
\end{proof}

\begin{lem}\label{l.hest}
Let $0 < a_0<a<b<b_0<1$. If $\Gamma (k,a_0,b_0)< \infty$ for some $k>1$ and $\rho \in L^r (Q(a_0,b_0))$ for some
$1 <r< \infty$, then $ \rho \in \mathscr{H}^{p,1}(Q(a,b))$
for $p=rk/(r+k-1)$.
\end{lem}

\begin{proof}
Throughout the proof, $c$ denotes a constant depending on $k, a_0, a, b, b_0, d$ and the coefficients $q_{ij}$, which may change from line to line.

Let $\eta$ be a smooth function such that $0\leq \eta \leq 1$, $\eta (s) =1$ for $s\in [a,b]$ and $\eta (s) = 0$ for
$s\leq a_0$ and for $s\geq b_0$.  Let $\varphi \in C^{1,2}_c(\bar{Q}(0,1))$.
Applying Lemma \ref{l.differential} to $\eta\varphi$, we obtain
\begin{equation}\label{eq.etaq}
\int_{Q(0,1)}\eta \rho (\partial_s\varphi - \A_0\varphi)\, ds \, dx = \int_{Q(0,1)}[ (F\cdot\nabla \varphi)\eta \rho -
\rho\varphi\partial_s\eta]\, ds\, dx\,.
\end{equation}

We infer from H\"older's inequality that
\begin{align*}
\int_{Q(a_0,b_0)} |F|^p\rho^p\, ds\, dx &= \int_{Q(a_0,b_0)} |F|^p\rho^\frac{p}{k} \rho^{p-\frac{p}{k}} \, ds\,dx\\
&\leq  \bigg( \int_{Q(a_0,b_0)} |F|^k\rho\, ds\,dx \bigg)^{\frac{p}{k}}
\bigg( \int_{Q(a_0,b_0)} \rho^\frac{p(k-1)}{k-p}\, ds\,dx \bigg)^{1-\frac{p}{k}}\\
&=  \Gamma (k,a_0,b_0)^p \bigg( \int_{Q(a_0,b_0)} \rho^r\,ds\,dx \bigg)^{1-\frac{p}{k}}\,.
\end{align*}
Consequently,
\begin{eqnarray*}
\||F|\rho\eta\|_{L^p(Q(a_0,b_0))} \leq \||F|\rho\|_{L^p(Q(a_0,b_0))} \leq c \|\rho\|_{L^r(Q(a_0,b_0))}^\frac{k-1}{k}\, .
\end{eqnarray*}
The same computation with $|F|$ replaced with $\one$ shows that $\rho \in L^p(Q(a_0,b_0))$ with
\begin{equation}
\|\rho\|_{ L^p(Q(a_0,b_0))} \leq \|\rho\|_{ L^r(Q(a_0,b_0))}^{\frac{k-1}{k}}\,.
\label{star}
\end{equation}
Combining these estimates
with \eqref{eq.etaq}, it follows that
\begin{equation}\label{eq.etaq2}
\bigg| \int_{Q(0,1)} \eta \rho (\partial_s\varphi - \A_0\varphi)\, ds \, dx\bigg| \leq c \|\rho\|_{ L^r(Q(a_0,b_0))}^{\frac{k-1}{k}}
\|\varphi\|_{W^{0,1}_{p'}(Q(0,1))}\,,
\end{equation}
where $1/p+1/p'=1$. In what follows, we write $\tilde{\rho}$ for $\eta \rho$.

Now fix $j_0 \in \{1, \ldots, d\}$. For a function $\psi = \psi (t,x)$ and $h \in \CR$ small enough, we denote
by $\Delta_h\psi$ the difference quotient
\begin{eqnarray*}
\Delta_{h}\psi (t,x) := h^{-1}\big[\psi (t,x+he_{j_0}) - \psi (t,x)\big]\,.
\end{eqnarray*}
An easy computation shows that
\begin{align*}
& \int_{Q(0,1)} \tilde{\rho}\big[\partial_s (\Delta_{-h}\varphi) - \A_0 (\Delta_{-h}\varphi )\big]\, ds\,dx\\
= &  - \int_{Q(0,1)} (\Delta_{h}\tilde{\rho})\big[\partial_s\varphi - \A_0 \varphi \big]\, ds\,dx\\
& + \int_{Q(0,1)}  \tilde{\rho}(s, x+he_{j_0}) \sum_{i,j=1}^d (\Delta_hq_{ij})(s,x)D_{ij}\varphi (s,x)\, ds\,dx\\
=: & - I_1 + I_2\, .
\end{align*}
Noting that $\Delta_hq_{ij}(s,x) = D_{j_0}q_{ij}(s, \xi_x)$ for some $\xi_x$ on the segment from $x$ to $x+he_{j_0}$
and that the first order derivatives of the diffusion coefficients are bounded on $Q(0,1)$, it follows that
\begin{eqnarray*}
|I_2| \leq c \|\rho\|_{L^r(Q(a_0,b_0))}^{\frac{k-1}{k}} \|\varphi\|_{W^{1,2}_{p'}(Q(0,1))}\, .
\end{eqnarray*}
Since $\|\Delta_h \varphi\|_{W^{0,1}_{p'}} \leq \|\varphi \|_{W^{1,2}_{p'}}$ as a consequence of the mean value theorem, the above combined with \eqref{eq.etaq2} yields
\begin{equation}\label{eq.etaq3}
\bigg| \int_{Q(0,1)} (\Delta_h\tilde{\rho}) \big[\partial_s\varphi - \A_0\varphi\big]\, ds \, dx\bigg| \leq c \|\rho\|_{ L^r(Q(a_0,b_0))}^{\frac{k-1}{k}}
\|\varphi\|_{W^{1,2}_{p'}(Q(0,1))}\,.
\end{equation}
Observing that $\tilde{\rho}$, and hence $\Delta_h\tilde{\rho}$, is an element of $L^p(Q(0,1))$, an approximation argument shows that \eqref{eq.etaq3} even holds for $\varphi \in W_{p'}^{1,2}(Q(0,1))$.

Now observe that $|\Delta_h\tilde{\rho}|^{p-2}\Delta_h\tilde{\rho}$ belongs to $L^{p'}(Q(0,1))$. As a consequence of \cite[Theorem IV.9.1]{lsu}
the Cauchy problem
\begin{eqnarray*}
\left\{\begin{array}{rlll}
\partial_su(s,x) - (\A_0(s)u(s))(x)& = & |\Delta_h\tilde{\rho}(s,x)|^{p-2}\Delta_h\tilde{\rho}(s,x)\,, & (s,x) \in Q(0,1)\,,\\[2mm]
u(0,x) & = & 0\,, & x \in \CR^d\,,
\end{array}
\right.
\end{eqnarray*}
has a unique solution $u \in W^{1,2}_{p'}(Q(0,1))$ and
\begin{eqnarray*}
\|u\|_{ W^{1,2}_{p'}(Q(0,1))}
\leq c \| |\Delta_h\tilde{\rho}|^{p-1}\|_{L^{p'}(Q(0,1))}\,.
\end{eqnarray*}
Inserting $u$ into \eqref{eq.etaq3} and using this estimate, we obtain
\begin{eqnarray*}
\int_{Q(0,1)}|\Delta_h\tilde{\rho}|^p\, ds\,dx \leq c \|\rho\|_{L^r(Q(a_0,b_0))}^{\frac{k-1}{k}} \|\Delta_h\tilde{\rho}\|_{L^p(Q(0,1))}^{p-1}\, .
\end{eqnarray*}
Thus
\begin{eqnarray*}
\|\Delta_h\tilde{\rho}\|_{L^p(Q(0,1))} \leq c \|\rho\|_{L^r(Q(a_0,b_0))}^{\frac{k-1}{k}} \, .
\end{eqnarray*}
Since the difference quotients $\Delta_h\tilde{\rho}$ are bounded in $L^p$, it follows from reflexivity that
they have a weak cluster point $g \in L^p(Q(0,1))$. Testing against a function in $C^\infty_c (Q(0,1))$, it follows that
$g$ is the weak derivative of $\tilde{\rho}$ in the direction in which the difference quotients were taken.
As this direction was arbitrary, it follows that
$\tilde{\rho} \in W^{0,1}_{p}(Q(0,1))$ and
\begin{equation}
\|\nabla \tilde{\rho}\|_{L^p(Q(0,1))}\le c\|\rho\|_{L^r(Q(a_0,b_0))}^{\frac{k-1}{k}}\,.
\label{estim-grad}
\end{equation}

Let us now consider the distributional time derivative of $\rho$. From \eqref{eq.etaq} we deduce that
\begin{eqnarray*}
\bigg |\int_{Q(0,1)}\tilde\rho \partial_s\varphi\, ds \, dx\bigg | = \bigg |\int_{Q(0,1)}[\tilde\rho\A_0\varphi+ (F\cdot\nabla \varphi)\tilde\rho -
\rho\varphi\partial_s\eta]\, ds\, dx\bigg |\,,
\end{eqnarray*}
for any $\varphi\in C^{1,2}_c(\bar{Q}(0,1))$.
Integrating by parts, we find
\begin{eqnarray*}
\int_{Q(0,1)}\tilde\rho\A_0\varphi\, ds\, dx
=-\int_{Q(0,1)}\langle Q\nabla_x\tilde\rho,\nabla_x\varphi\rangle \, ds\, dx
-\int_{Q(0,1)}\tilde\rho\sum_{i,j=1}^dD_iq_{ij}D_j\varphi\, ds\, dx\,.
\end{eqnarray*}
Using \eqref{star}, \eqref{estim-grad} and the boundedness of the spatial derivatives of the $q_{ij}$, we deduce that
\begin{align*}
\bigg |
\int_{Q(0,1)}\tilde\rho\A_0\varphi\, ds\, dx\bigg | & \le c\|\tilde\rho\|_{W^{0,1}_p(Q(0,1))}\|\nabla_x\varphi\|_{L^{p'}(Q(0,1))}\\
&\le  c\|\rho\|_{L^r(Q(a_0,b_0))}^{\frac{k-1}{k}}\|\nabla_x\varphi\|_{L^{p'}(Q(0,1))}\, .
\end{align*}
Combining this with the above estimates yields
\begin{eqnarray*}
\bigg |\int_{Q(0,1)}\tilde\rho \partial_s\varphi\, ds \, dx\bigg |
\le c\|\rho\|_{L^r(Q(a_0,b_0))}^{\frac{k-1}{k}}\|\varphi\|_{W^{0,1}_{p'}(Q(0,1))}\,,
\end{eqnarray*}
which implies that $\partial_s \tilde\rho\in (W^{0,1}_{p'}(Q(0,1)))'$. Altogether, we have showed that
 $u\in \mathscr{H}^{p,1}(Q(0,1))$. Since $\tilde\rho\equiv\rho$ in
$Q(a,b)$, the assertion follows.
\end{proof}

We can now iterate Lemma \ref{l.hest} to obtain better regularity of $\rho$.

\begin{prop}\label{p.hest}
Let $0 < a_0 < a<b<b_0 < 1$. If $\Gamma (k,a_0,b_0) < \infty$ for some $1<k\le d+2$, then $\rho \in L^r(Q(a,b))$ for
all $r \in [1, (d+2)/(d+2-k))$ and $\rho \in \mathscr{H}^{p,1}(Q(a,b))$ for all
$p \in (1, (d+2)/(d+3-k))$.
\end{prop}

\begin{proof}
Fix a parameter $m$ which will be specified later on and define $a_n := a_0 + n(a-a_0)/m$ and $b_n := b_0 - n(b_0-b)/m$ for $n=1, \ldots,
m$.
Picking $r_1 \in (1, (d+2)/(d+1))$, it follows from Proposition \ref{p.qinlr} that $\rho \in L^{r_1}((a_1,b_1)\times \CR^d)$.
Now assume that $\rho \in L^{r_n}((a_n,b_n) \times \CR^d)$. It follows from Lemma \ref{l.hest}, that
$\rho \in \mathscr{H}^{p_n,1}((a_{n+1}, b_{n+1})\times \CR^d)$, where
$p_n = r_nk/(r_n+k-1)$. By \cite[Theorem 7.1]{mpr10} $\mathscr{H}^{p_n,1}((a_{n+1}, b_{n+1})\times \CR^d)$ is continuously embedded into
$L^{r_{n+1}}((a_{n+1}, b_{n+1})\times \CR^d)$, where
\begin{eqnarray*}
\frac{1}{r_{n+1}} = \frac{1}{p_n} - \frac{1}{d+2} = \frac{1}{r_n}\bigg (1-\frac{1}{k}\bigg ) + \frac{1}{k} - \frac{1}{d+2}\, .
\end{eqnarray*}
Since $r_1^{-1}> (d+1)/(d+2)$, we see that
\begin{eqnarray*}
\frac{1}{r_2} - \frac{1}{r_1} < - \frac{1}{k}\bigg( 1 - \frac{1}{d+2}\bigg ) + \frac{1}{k} - \frac{1}{d+2}= \frac{1}{d+2}\bigg( \frac{1}{k} - 1\bigg) < 0\, .
\end{eqnarray*}
Proceeding inductively, we see that $\frac{1}{r_n}$ is a decreasing sequence of positive numbers, hence it is convergent. Its limit is easily seen
to be $(d+2-k)/(d+2)$.

Therefore, given $r< (d+2)/(d+2-k)$, after finitely many steps we have $r_n>r$. The number of steps needed is our parameter $m$.
Thus, after $m$ steps, we have $\rho \in L^r(Q(a,b))$. The assertion concerning $\mathscr{H}^{p,1}$ follows from
Lemma \ref{l.hest}.
\end{proof}

\begin{cor}\label{c.boundedness}
Let $0 < a_0<a<b<b_0<1$. If $\Gamma (k, a_0,b_0) < \infty$ for some $k> d+2$, then $\rho$ belongs to
 $\mathscr{H}^{p,1}$, for some $p>d+2$, and also to $L^\infty(Q(a,b))$.
\end{cor}

\begin{proof}
Let $a_0<a_1<a$ and $b<b_1<b_0$. Since, by H\"older's inequality, $\Gamma(d+2,a_0,b_0)\le \Gamma (k,a_0,b_0)<\infty$,
it follows from Proposition \ref{p.hest} that $\rho \in L^r((a_1,b_1)\times \CR^d)$ for all $r \in [1,\infty)$. Thus, by Lemma \ref{l.hest},
$\rho \in \mathscr{H}^{p,1}(Q(a,b))$ for all $p\in (1, k)$. In particular, $\rho \in \mathscr{H}^{p_0,1}(Q(a,b))$ for
some $p_0\in (d+2,k)$. By \cite[Theorem 7.1]{mpr10},
 $\mathscr{H}^{p_0,1}(Q(a,b))$ is continuously embedded into $L^{\infty}(Q(a,b))$.
\end{proof}

\subsection{Boundedness of weak solutions to nonautonomous parabolic problems} We next consider functions $u$ which are, in some sense,
weak solutions to an inhomogeneous parabolic equation $\partial_t u - \A(t)u = f$ and provide an estimate of their supremum norm.
A related result was proved in \cite[Theorem 7.3]{mpr10}.
Note, however, that the constant obtained in that result depends on the sup norm of the diffusion coefficients. Therefore \cite[Theorem 7.3]{mpr10}
cannot be used in the approximation argument in the next section.
For this reason, we present a different result here which is a parabolic version of \cite[Theorem A.1]{ffmp09}, where the elliptic equation
was considered.\smallskip

Before stating and proving the main result of this subsection, we need some preparation. First, we recall an embedding result from Chapter 2, \S 3 of \cite{lsu} and prove an integration by parts formula.

\begin{lem}\label{l.embed}
Let $d \geq 2$, $p$ and $q$ be given such that $\frac{1}{p} + \frac{d}{2q} = \frac{d}{4}$. Here we have $p \in [2, \infty]$ and $q \in [2, 2d/(d-2)]$
in the case where $d \geq 3$ and $p \in (2,\infty]$, $q \in [2,\infty)$ in the case where $d=2$.
Then every function in $W^{0,1}_2(Q(a,b))\cap L^\infty(a,b; L^2(\CR^d))$ belongs to $L^p(a,b; L^q(\CR^d))$.
Moreover,  there is a constant $c_S$, which is independent of $a,b$ in bounded subsets of $\CR$,
 such that for $u \in  W^{0,1}_2(Q(a,b))\cap L^\infty(a,b; L^2(\CR^d))$ we have
\begin{eqnarray*}
\|u\|_{p,q} \leq c_S (\|u\|_{\infty, 2} + \|\nabla u\|_2)\, .
\end{eqnarray*}
\end{lem}

\begin{lem}
\label{lem-cortona}
Let $u\in \mathscr{H}^{p,1}(Q(a,b))$ for some $p>d+2$, $\ell>0$ and $\vartheta:\CR^d\to\CR$ be a nonnegative, smooth and compactly supported function.
Then, $\vartheta (u-\ell)_+ \in W^{0,1}_{p'}(Q(a,b))$ and
\begin{equation}
\int_{Q(a,b)}\vartheta(u-\ell)_+\partial_tu\,dt\,dx
=\frac{1}{2}\left[\int_{\CR^d}\vartheta((u(b)-\ell)_+)^2\,dx-\int_{\CR^d}\vartheta((u(a)-\ell)_+)^2\,dx\right]\,.
\label{AA0}
\end{equation}
\end{lem}

\begin{proof}
We observe first that $\vartheta (u-\ell)_+ \in W^{0,1}_{p'}(Q(a,b))$, since
$\vartheta (u-\ell)_+=\left(\vartheta (u-\ell)\right)_+$ and $\vartheta (u-\ell)\in W^{0,1}_{p'}(Q(a,b))$ (cf. \cite[Lemma 7.6]{GT}).

Let $(u_n)\in C^{\infty}_c(\CR^{d+1})$ be a sequence converging to $u$ in the $\mathscr{H}^{p,1}$-norm (see \cite[Lemma 7.1]{mpr10}). Since $p>d+2$, it follows from \cite[Theorem 7.1]{mpr10} that ${\mathscr H}^{p,1}(Q(a,b))$ is continuously embedded into $L^{\infty}(Q(a,b))$. Hence,
$u_n$ converges to $u$ uniformly in $Q(a,b)$. So, we can deduce that $\vartheta (u_n-\ell)_+$ tends to $\vartheta (u-\ell)_+$ in $W_{p'}^{0,1}(Q(a,b))$. Hence,
\begin{align}
\int_{Q(a,b)}\vartheta(u-\ell)_+\partial_tu\,dt\,dx=\lim_{n\to \infty}\int_{Q(a,b)}\vartheta(u_n-\ell)_+\partial_tu\,dt\,dx
\label{AA}
\end{align}
follows since $\partial_t u\in (W_{p'}^{0,1}(Q(a,b)))'$.

We now split
\begin{align}
\int_{Q(a,b)}\vartheta(u_n-\ell)_+\partial_tu\,dt\,dx & =
\int_{Q(a,b)}\vartheta(u_n-\ell)_+\partial_tu_n\,dt\,dx\notag\\
&\quad+\int_{Q(a,b)}\vartheta(u_n-\ell)_+(\partial_tu-\partial_tu_n)\,dt\,dx\notag\\
&=\frac{1}{2}\int_{Q(a,b)}\vartheta \partial_t((u_n-\ell)_+^2)\,dt\,dx\notag\\
&\quad +\int_{Q(a,b)}\vartheta(u_n-\ell)_+(\partial_tu-\partial_tu_n)\,dt\,dx\notag\\
&=\frac{1}{2}\int_{\CR^d}\vartheta[(u_n(b)-\ell)_+^2-(u_n(a)-\ell)_+^2]\,dx\notag\\
&\quad +\int_{Q(a,b)}\vartheta(u_n-\ell)_+(\partial_tu-\partial_tu_n)\,dt\,dx.
\label{BB}
\end{align}
We claim that
\begin{equation}
\left\{
\begin{array}{l}
\displaystyle \mathrm{(i)}~\lim_{n\to\infty}
\int_{Q(a,b)}\vartheta (u_n-\ell)_+(\partial_tu-\partial_tu_n)\,dt\,dx=0,\\[5mm]
\displaystyle  \mathrm{(ii)}~\lim_{n\to\infty}\int_{\CR^d}\vartheta (u_n(s)-\ell)_+^2\,dx=
\int_{\CR^d}\vartheta (u(s)-\ell)_+^2\,dx,\qquad\;\,s\in [a,b]\,.
\end{array}
\right.
\label{CC}
\end{equation}
Property (i) follows immediately from the boundedness of the sequence $(\vartheta (u_n-\ell)_+)$ in
$W^{0,1}_{p'}(Q(a,b))$ and the convergence of $\partial_tu_n$ to $\partial_tu$ in $(W^{0,1}_{p'}(Q(a,b)))'$.
Property (ii) follows from the uniform convergence of $u_n$ to $u$ in $Q(a,b)$ and the fact that $\vartheta$
is compactly supported in $\CR^d$.

From \eqref{AA}, \eqref{BB} and \eqref{CC}, we get \eqref{AA0}.
\end{proof}

The following is the main result of this subsection.

\begin{thm}\label{t.auxiliary}
Let $q_{ij} \in C^{\frac{\varsigma}{2}, \varsigma}_\loc([0,1]\times \CR^d)\cap C_b([0,1]\times \CR^d)$ be such that
$q_{ij} = q_{ji}$ for $i,j=1, \ldots, d$ and such that $\langle Q(t,x) \xi, \xi\rangle \geq \eta |\xi|^2$ for a certain $\eta >0$
and any $(t,x) \in [0,1]\times \CR^d, \xi \in \CR^d$.

Further, let $0\leq a_0 < b_0 \leq 1$, $r>d+2$
and let functions  $f \in L^\frac{r}{2}(Q(a_0,b_0))$,
$h= (h_i) \in L^r(Q(a_0,b_0), \CR^d)$ and $u \in \mathscr{H}^{p,1}(Q(a_0,b_0))$, for some $p>d+2$, be given such that $u(b_0) \equiv 0$ and
\begin{equation}\label{eq.weak}
\int_{Q(a_0,b_0)}\big [\langle Q\nabla u,\nabla \varphi\rangle - \varphi \partial_t u\big ] \, dt\,dx
= \int_{Q(a_0,b_0)} f\varphi \, dt\,dx + \int_{Q(a_0,b_0)}\langle h, \nabla \varphi\rangle \, dt\,dx\,,
\end{equation}
for all $\varphi \in C_c^\infty (\bar{Q}(a_0,b_0))$. Then, $u$ is bounded and there exists a constant $C>0$, depending only on
$\eta,d$ and  $r$ (but \emph{not} depending on $\|Q\|_{\infty}$) such that
\begin{eqnarray*}
\|u\|_{\infty}\leq C \big( \|u\|_{\infty, 2} + \|f\|_{\frac{r}{2}} + \|h\|_{r}\big)\,.
\end{eqnarray*}
\end{thm}

\begin{proof}
A density argument shows that \eqref{eq.weak} is also satisfied by $\varphi \in W^{0,1}_2(Q(a,b))$  such that there exists $R>0$ with $\varphi (t,x) = 0$ for all $t \in (a,b)$ and $|x|> R$.\smallskip

We first additionally assume that $\|u\|_{\infty, 2}, \|f\|_{\frac{r}{2}}, \|h\|_{r} \leq 1$. Note that $u\in L^{\infty}(Q(a_0,b_0))$ since $p>d+2$ (see \cite[Theorem 7.1]{mpr10}).

Fix $\ell>1$. Then $(u-\ell)_+ \in W^{0,1}_p(Q(a_0,b_0))$. It also belongs to $W^{0,1}_2(Q(a_0,b_0))$ since
$u\in L^{\infty}(Q(a_0,b_0))$, $\nabla (u-\ell)_+=\nabla u\one_{\{u\ge\ell\}}$ and $\one_{\{u\ge\ell\}}\in L^q(Q(a_0,b_0))$ for any $q\in [1,\infty]$.

If $\vartheta_n$ is a standard sequence of cutoff functions (in $x$), we may plug $\varphi := \vartheta_n^2(u-\ell)_+$ into \eqref{eq.weak}.
Thus, taking Lemma \ref{lem-cortona} into account and observing that $(u(b_0)-\ell)_+ \equiv 0$, \eqref{eq.weak} becomes
\begin{eqnarray*}
\begin{aligned}
&\phantom{=} \half \int_{\CR^d}\vartheta_n^2 (u(a_0)-\ell)_+^2\, dx + \int_{Q(a_0,b_0)} \vartheta_n^2 \langle Q \nabla (u-\ell)_+, \nabla (u-\ell)_+\rangle\, dt\,dx\\
&\qquad + 2\int_{Q(a_0,b_0)}\vartheta_n \langle Q\nabla (u-\ell)_+, \nabla \vartheta_n\rangle (u-\ell)_+\,dt\,dx\\
&=  \int_{Q(a_0,b_0)} f\vartheta_n^2 (u-\ell)_+\,dt\,dx + \int_{Q(a_0,b_0)} \vartheta_n^2\langle h,\nabla (u-\ell)_+\rangle\,dt\,dx\\
 &\qquad + 2\int_{Q(a_0,b_0)}\vartheta_n (u-\ell)_+\langle h, \nabla \vartheta_n\rangle \,dx\,.
\end{aligned}
\end{eqnarray*}

We have
$|\langle Q\nabla (u-\ell)_+, \nabla \vartheta_n\rangle  | \leq |Q^\half \nabla (u-\ell)_+ | |Q^\half \nabla \vartheta_n|$
by the Cauchy-Schwarz inequality.
Thus,
\begin{align*}
&\phantom{=} 2\Big|\int_{Q(a_0,b_0)}\vartheta_n \langle Q\nabla (u-\ell)_+, \nabla \vartheta_n\rangle (u-\ell)_+\,dt\,dx \Big|\\
&\leq   \frac{1}{2}\int_{Q(a_0,b_0)} \vartheta_n^2 \langle Q\nabla (u-\ell)_+, \nabla (u-\ell)_+\rangle\,dt\,dx
+ \frac{c\|Q\|_\infty}{n^2} \int_{Q(a_0,b_0)} (u-\ell)_+^2\,dt\,dx\,,
\end{align*}
for some positive constant $c$ independent of $n$.

We now put $A_{\ell}(t) := \{ u(t) \geq \ell\}$, $A_{\ell} := \{u\geq \ell\}$ and write $|A_{\ell}(t)|$ for the $d$-dimensional Lebesgue measure of $A_{\ell}(t)$
and $|A_{\ell}|$ for the $d+1$-dimensional Lebesgue measure of $A_{\ell}$.\smallskip

Let us now estimate the term involving $f$. Applying H\"older's inequality with exponent $\frac{r}{2},\,2+\frac{4}{d}$ and $s$ with $\frac{1}{s}=\frac{1}{2}-\frac{2}{r}+\frac{1}{d+2}$,
we obtain
\begin{eqnarray*}
\int_{Q(a_0,b_0)} \vartheta_n^2|f| (u-\ell)_+\, dt\, dx
 \leq  \|\vartheta_n f\|_{\frac{r}{2}} \| (u-\ell)_+\|_{2+\frac{4}{d}}|A_\ell|^{\frac{1}{2}-\frac{2}{r}+\frac{1}{d+2}}\,.
\end{eqnarray*}
So, by Lemma \ref{l.embed} with $p=q=2+\frac{4}{d}$, and since $\|f\|_{\frac{r}{2}}\le 1$, it follows that
\begin{equation}\label{eq.fest}
 \int_{Q(a_0,b_0)} \vartheta_n^2|f| (u-\ell)_+\, dt\, dx
 \leq  c_S \big(\|(u-\ell)_+\|_{\infty, 2} + \|\nabla (u-\ell)_+\|_2 \big) |
A_\ell|^{\frac{1}{2}-\frac{2}{r}+\frac{1}{d+2}}\,.
\end{equation}

As for the integral involving $h$, H\"older's inequality yields
\begin{equation}\label{eq.hest}
\int_{Q(a_0,b_0)}\vartheta_n^2 |h| |\nabla (u-\ell)_+|\, dt \, dx
\leq  \|h\|_{r} \|\nabla (u-\ell)_+\|_2 |A_\ell|^{\half - \frac{1}{r}}\,.
\end{equation}
Note that estimates \eqref{eq.fest} and \eqref{eq.hest} imply together with monotone convergence that the integrals
$\int_{Q(a_0,b_0)} f(u-\ell)_+\,dt\,dx$ and $\int_{Q(a_0,b_0)}\langle h,\nabla (u-\ell)_+\rangle\,dt\,dx$ exist. Finally, we estimate
\begin{eqnarray*}
\bigg| 2\int_{Q(a_0,b_0)} \vartheta_n (u-\ell)_+ \langle h,\nabla \vartheta_n\rangle\, dt\,dx\bigg| \leq
\frac{C}{n} \|h\|_{r} \|(u-\ell)_+\|_2 |A_\ell|^{\frac{1}{2} - \frac{1}{r}}\,.
\end{eqnarray*}
Collecting the estimates and letting $n\to \infty$, we obtain
\begin{align*}
&\phantom{=} \half \int_{\CR^d} (u(a_0,x)-\ell)_+^2\, dx + \half \int_{Q(a_0,b_0)} \langle Q \nabla (u-\ell)_+, \nabla (u-\ell)_+ \rangle \, dt\,dx\\
&\leq  \int_{Q(a_0,b_0)} |f|(u-\ell)_+\, dt\,dx +  \int_{Q(a_0,b_0)} |h||\nabla (u-\ell)_+|\, dt\,dx\,.
\end{align*}
Observe that we can repeat the above arguments with $a_0$ replaced with an element $a_0' \in (a_0,b_0)$. Taking the supremum over such $a_0'$ and using the estimates from \eqref{eq.fest} and \eqref{eq.hest}, as well as the ellipticity assumption, we obtain
\begin{align*}
& \phantom{=} \min\left\{1, \eta\right\} \big(\|(u-\ell)_+\|_{\infty, 2}^2 +\|\nabla (u-\ell)_+\|_2^2\big)\\
& \leq 2c_S\big( \|(u-\ell)_+\|_{\infty, 2}+\|\nabla (u-\ell)_+\|_2\big)
|A_\ell|^{\frac{1}{2}-\frac{2}{r}+\frac{1}{d+2}} + 2\|\nabla (u-\ell)_+\|_2 |A_\ell|^{\frac{1}{2} -\frac{1}{r}}\, .
\end{align*}
Noting that $\half -\frac{1}{r} < \frac{1}{2}-\frac{2}{r}+\frac{1}{d+2}$, by our assumption on $r$ and $\ell$, and that $|A_\ell| \leq 1$,
it follows that
\begin{equation}\label{eq.normest}
\|(u-\ell)_+\|_{\infty, 2} +\|\nabla (u-\ell)_+\|_2 \leq L |A_\ell|^{\half - \frac{1}{r}}\,,
\end{equation}
for a certain constant $L$. For $m>\ell$ we find
\begin{align}
(m-\ell)^2|A_m| & \leq  \int_{A_m} (u-\ell)_+^2\, dt\,dx\notag\\
&\leq \int_{A_\ell} (u-\ell)_+^2 \, dt\,dx\notag\\
&\leq  \|(u-\ell)_+\|_{2+\frac{4}{d}}^2 \|\one_{A_\ell}\|_{\frac{d+2}{2}}\notag\\
&\leq  c_S \big( \|(u-\ell)_+\|_{\infty,2} + \|\nabla (u-\ell)_+\|_2\big)^2|A_\ell|^{\frac{2}{d+2}}\notag\\
& \leq  Lc_S |A_\ell|^{1-\frac{2}{r}+\frac{2}{d+2}} =: \nu_d|A_\ell|^{1-\frac{2}{r} + \frac{2}{d+2}}\, .\label{eq.nuest}
\end{align}
Here we have used H\"older's inequality with exponents $1+\frac{2}{d}$ and $\frac{d+2}{2}$, Lemma \ref{l.embed} and estimate
\eqref{eq.normest} above.

Now, the proof can be completed following the same arguments as in
\cite[Theorem A.1]{ffmp09}. For the reader's convenience we provide all the details.
Let $\bar{\ell}\geq 1$ and define $\ell_n := 2\bar{\ell} - 2^{-n}\bar{\ell}$ for $n \in \CN$ and $y_n := |A_{\ell_n}|$. Using \eqref{eq.nuest}
 for $m=\ell_{n+1}$ and $\ell=\ell_n$, it follows that
\begin{eqnarray*}
y_{n+1} \leq \frac{4\nu_d}{\bar{\ell}^2} 2^{2n}y_n^{1+\alpha}\,,
\end{eqnarray*}
where $\alpha = \frac{2}{d+2} - \frac{2}{r}> 0$. If $y_0 \leq \big(\frac{4\nu_d}{\bar{\ell}^2}\Big)^{-\frac{1}{\alpha}}4^{-\frac{1}{\alpha^2}}$, then $y_n\to 0$ as $n\to \infty$ (cf. \cite[Lemma 7.1]{giusti03}) which implies that $|A_l| = 0$ for $l \geq 2\bar{\ell}$, i.e.\ $u \leq 2\bar{\ell}$. As $y_0 = |A_{\bar{\ell}}| \leq 1$, this estimate
holds if we pick $\bar{\ell} = \max\{ 1, 2^{1+\frac{1}{\alpha}}\sqrt{\nu_d}\}=: C/2$. Thus, $u \leq C$. Replacing $u$ with $-u$, by linearity
we obtain that also $-u\leq C$, hence $\|u\|_{\infty} \leq C$.\smallskip

We finally remove the additional assumption that $\|u\|_{\infty, 2}, \|f\|_{\frac{r}{2}}, \|h\|_{r} \leq 1$. To that end, let
$M := \|u\|_{\infty, 2} + \|f\|_{\frac{r}{2}} +\|h\|_{r}$ and define $\tilde{f} := f/M, \tilde{h}:= h/M$ and $\tilde{u} := u/M$.
Clearly, \eqref{eq.weak} holds for all $\varphi \in C_c^\infty(Q(a,b))$ when $u,f,h$ are replaced with
$\tilde{u},\tilde{f},\tilde{h}$. By the above $\|\tilde{u}\|_{\infty} \leq C$, thus $\|u\|_{\infty} \leq CM$ which is the thesis.
\end{proof}

\section{Pointwise estimates for the densities of evolution systems of measures}

In this section we consider the following assumptions.
\begin{hyp}\label{hyp2}
Let $t \in (0,1]$ and $(\mu_s)_{s \in [0,t]}$ be an evolution system of measures, such that $V \in L^1(\mu_t)$, and let $W_1$, $W_2$ be time dependent
Lyapunov functions for the operator $\partial_s - \A(s)$ on $[0,t]$ with $W_1 \leq W_2\le c_0V^{1-\sigma}$ for some constant $c_0>0$ and $\sigma\in (0,1)$. Moreover, let $0<a_0<a<b<b_0<t$ and
$1 \leq w \in C^{1,2}(Q(0,t))$ be a weight function such that the following holds true:
\begin{enumerate}
\item $w^{-2}\partial_sw$ and $w^{-2}\nabla w$ are bounded on $[a_0,b_0]\times \CR^d$;
\item there exist constants $c_1, \ldots, c_6$, possibly depending on the interval $(a_0,b_0)$, such that
\begin{eqnarray*}
\begin{array}{ll}
\mathrm{(i)} \quad w \leq c_1W_1\,, & \mathrm{(ii)} \quad |Q\nabla w| \leq c_2 w^{\frac{k-1}{k}}W_1^\frac{1}{k}\,,\\
 \mathrm{(iii)} \quad |\A_0 w| \leq c_3w^{\frac{k-2}{k}}W_1^\frac{2}{k}\,,
& \mathrm{(iv)} \quad |\partial_sw|\leq c_4w^{\frac{k-2}{k}}W_1^\frac{2}{k}\,,\\
 \mathrm{(v)} \quad  |\sum_{i=1}^d D_iq_{ij}| \leq c_5 w^{-\frac{1}{k}}W_2^\frac{1}{k}\,,\quad
& \mathrm{(vi)} \quad  w|F|^k \leq c_6W_2\,,
\end{array}
\end{eqnarray*}
on $[a_0,b_0]\times \CR^d$, where $k$ is any positive constant greater than $d+2$;
\item there exist constants $c_7, c_8$, possibly depending on the interval $(a_0,b_0)$, such that
$|\Delta w|\leq c_7 w^\frac{k-2}{k}W_1^\frac{2}{k}$ and $|Q\nabla W_1|\leq c_8w^\frac{k-1}{k}W_2^\frac{1}{k}$
on $[a_0,b_0] \times \CR^d$.
\end{enumerate}
\end{hyp}

In the situation of Hypothesis \ref{hyp2}, we write
\begin{eqnarray*}
\zeta_i(s) := \int_{\CR^d} W_i(s)\, d\mu_s\,,\qquad\;\,i=1,2\, .
\end{eqnarray*}
Note that $\zeta_i(s) < \infty$ for $s \in [0,t]$ and $i=1,2$ by Proposition \ref{p.lyapunov}.
The Borel measure $\nu$ is defined by \eqref{eq.nu} and $\rho$ denotes the density of $\nu$ with respect to $(d+1)$-dimensional Lebesgue measure.

In this section, we prove pointwise estimates for $\rho$ provided that Hypothesis \ref{hyp2} is satisfied. This is done by treating
the case of bounded diffusion coefficients first. Here only parts (1) and (2) of Hypothesis \ref{hyp2} are needed. We then prove  pointwise estimates
in the general case by approximation, making use of part (3) of Hypothesis \ref{hyp2}.
\medskip

As far as bounded diffusion coefficients are concerned, the following is our main result.

\begin{thm}\label{main-bounded}
Assume Hypothesis \ref{hyp2}(1)-(2) and that the diffusion coefficients of the operator $\A$ and their first order spatial derivatives are bounded. Then there exists a positive constant $C$, depending only on $\eta$ and $d$, such that
\begin{align}
w\rho \leq C \bigg [ & c_1 \sup_{s \in (a_0,b_0)} \zeta_1(s)
+ (c_2^{k}  + c_5^{k}+ c_6)\int_{a_0}^{b_0} \zeta_2(s)\, ds \notag \\
&  +
\Big( \frac{k^kc_1^2}{(b_0-b)^{k}} +c_2^{2k} + c_3^{k}+c_4^{k}\bigg) \bigg( \int_{a_0}^{b_0} \zeta_1(s)\, ds \bigg)^2 +c_2^{k}c_6\bigg( \int_{a_0}^{b_0} \zeta_2(s)\, ds \bigg)^2\notag \\
& +
\bigg( \frac{k^2c_1^\frac{4}{k}}{(b_0-b)^{2}}+c_2^4+c_3^2+ c_4^{2}\bigg) \bigg( \int_{a_0}^{b_0} \zeta_1(s)\, ds \bigg)^\frac{4}{k} +c_2^2c_6^\frac{2}{k}\bigg( \int_{a_0}^{b_0}\hskip -.5mm \zeta_2(s)\, ds \bigg)^\frac{4}{k}\bigg]\,,
\label{eq.final}
\end{align}
in $(a,b)\times\CR^d$.
\end{thm}
\begin{proof}
We first prove estimate \eqref{eq.final},  assuming additionally that $w$, along with its first order partial derivatives, is bounded.
To that end, observe that
\begin{align*}
\Gamma (k,a_0,b_0) &=  \int_{Q(a_0,b_0)} |F(s,x)|^k \rho(s,x)\, ds \, dx\\
&\leq \int_{Q(a_0,b_0)} w(s,x) |F(s,x)|^k \rho(s,x)\, ds \, dx\\
&\leq  c_6\int_{Q(a_0,b_0)} W_2(s,x)\rho(s,x)\, ds \, dx\\
& =  c_6 \int_{a_0}^{b_0} \zeta_2(s) \, ds <\infty\,,
\end{align*}
as a consequence of Proposition \ref{p.lyapunov}. Since the diffusion coefficients are bounded, it follows from Proposition \ref{p.hest}
and Corollary \ref{c.boundedness} that $\rho \in \mathscr{H}^{p,1}(Q(a_0,b_0))$ for some $p>d+2$ and that $\rho \in L^\infty (Q(a,b))$.

Let now $\vartheta$ be a smooth function with $\vartheta (s) = 1$ for $s \in [a,b]$ and $\vartheta (s) = 0$ for $s \geq b_0$, with
$|\vartheta'| \leq 2(b_0-b)^{-1}$. Moreover, let $\psi \in C^{1,2}_c(Q(a_0,b_0))$ and put $\varphi (s,x) := \vartheta (s)^\frac{k}{2}w(s,x)\psi (s,x)$.
By Lemma \ref{l.differential} we have
\begin{equation}\label{eq.start}
\int_{Q(a_0,b_0)}\big[ \partial_s\varphi (s,x) -\A (s)\varphi (s,x)\big]\rho(s,x)\, ds \, dx  = 0\, .
\end{equation}
We write $\tilde\rho := \vartheta^\frac{k}{2}\rho$. Observe that, since $w$ and its derivatives are bounded, $w\tilde\rho$ belongs to $W^{0,1}_p(Q(a_0,b_0))$ for some $p>d+2$. Using approximation and integration by parts, we see that $w\tilde\rho \in \mathscr{H}^{p,1}(Q(a_0,b_0))$and
\begin{eqnarray*}
\int_{Q(a_0,b_0)}\partial_s\psi \big( w\tilde\rho\big)\,ds\,dx=-\int_{Q(a_0,b_0)}\psi \partial_s\big( w\tilde\rho\big)\,ds\,dx\,.
\end{eqnarray*}
Thus,
\begin{align*}
\int_{Q(a_0,b_0)}\partial_s\big(\vartheta^{\frac{k}{2}}w\psi\big)\rho \,ds\,dx
& =  \frac{k}{2}\int_{Q(a_0,b_0)}\vartheta^{\frac{k-2}{2}}\vartheta'w\psi\rho\,ds\,dx
+\int_{Q(a_0,b_0)}(\partial_sw)\psi\tilde\rho\,ds\,dx\\
&\quad -\int_{Q(a_0,b_0)}\psi \partial_s\big( w\tilde\rho\big)\,ds\,dx\,.
\end{align*}
Moreover, since
\begin{eqnarray*}
\A(\psi w)= \A(w)\psi  +
w \langle F,\nabla\psi\rangle
+ \A_0(\psi)w + 2\langle Q\nabla \psi,\nabla w\rangle
\end{eqnarray*}
(where, as usual, $\A_0$ denotes the leading part of the operator $\A$),
integrating by parts the term $\A_0(\psi)w$ we finally get
\begin{align*}
& \phantom{=} \int_{Q(a_0,b_0)}\A(\vartheta^{\frac{k}{2}}\psi w)\rho\,ds\,dx=
\int_{Q(a_0,b_0)}\A(\psi w)\tilde\rho\,ds\,dx\\
&=\int_{Q(a_0,b_0)}\A(w)\psi\tilde\rho\,ds\,dx
+\int_{Q(a_0,b_0)}\langle F,\nabla \psi\rangle\,w\tilde\rho\,ds\,dx\\
&\qquad+2\int_{Q(a_0,b_0)}\langle Q\nabla w,\nabla \psi\rangle \tilde\rho\,ds\,dx
-\int_{Q(a_0,b_0)}\bigg (\sum_{i,j=1}^dD_iq_{ij}D_j\psi\bigg )w\tilde\rho\,ds\,dx\\
&\qquad -\int_{Q(a_0,b_0)}\langle Q\nabla (w\tilde\rho),\nabla\psi\rangle\,ds\,dx\,.
\end{align*}
Inserting these expressions into \eqref{eq.start} and rearranging, we find
\begin{align*}
&  \phantom{=}\int_{Q(a_0,b_0)}\big [\langle Q\nabla (w\tilde{\rho}), \nabla \psi\rangle - \psi\partial_s (w\tilde{\rho})\big ] \, ds\,dx\\
& =  \int_{Q(a_0,b_0)} \tilde{\rho} \bigg( 2 \sum_{i,j=1}^d q_{ij}(\partial_iw)(\partial_j \psi) - \sum_{i,j=1}^d w (D_iq_{ij})(D_j\psi)+ w \langle F,\nabla \psi\rangle\bigg) \, ds\,dx\\
& \qquad-\frac{k}{2}\int_{Q(a_0,b_0)}\rho w\psi \vartheta^{\frac{k-2}{k}}\vartheta'\, ds\,dx\\
& \qquad + \int_{Q(a_0,b_0)} \tilde{\rho} (\psi\A_0 w+\psi\langle F,\nabla w\rangle - \psi \partial_s w)\, ds\,dx\,.
\end{align*}
We may thus invoke Theorem \ref{t.auxiliary} and obtain
\begin{equation}\label{eq.inftyest}
\begin{aligned}
\|w\tilde{\rho} \|_{\infty}\leq
C \bigg( &\|w\tilde{\rho}\|_{\infty,2} +
\|\tilde{\rho} Q\nabla w\|_{k} +\|w\tilde{\rho} F\|_{k}
+ \sum_{j=1}^d\bigg\|w\tilde{\rho}\sum_{i=1}^d D_iq_{ij}\bigg\|_{k}\\
&+  \frac{k}{b_0-b}\|w\rho\vartheta^{\frac{k-2}{k}}\|_{\frac{k}{2}}
+ \|\tilde{\rho}\A_0w \|_{\frac{k}{2}}
  + \|\tilde{\rho}\partial_sw \|_{\frac{k}{2}}
+ \|\tilde{\rho}F\cdot\nabla w \|_{\frac{k}{2}}  \bigg)\,,
\end{aligned}
\end{equation}
where $C$ is a constant depending only on $k, d$ and $\eta$.
Next, observe that
\begin{align*}
\|\tilde{\rho} Q\nabla w\|_{k} &=
\bigg( \int_{Q(a_0,b_0)} |\tilde{\rho}Q\nabla w|^k \, ds \,dx \bigg)^\frac{1}{k}\\
&\leq  c_2\bigg( \int_{Q(a_0,b_0)} \tilde{\rho}^k w^{k-1} W_1\, ds \,dx \bigg)^\frac{1}{k}\\
&\leq  c_2 \|w\tilde{\rho}\|_{\infty}^{\frac{k-1}{k}} \bigg( \int_{Q(a_0,b_0)} \tilde{\rho} W_1\, ds \,dx \bigg)^\frac{1}{k}\\
&\leq  c_2 \|w\tilde{\rho}\|_{\infty}^{\frac{k-1}{k}} \bigg( \int_{a_0}^{b_0} \zeta_1(s) \, ds \bigg)^\frac{1}{k}\,.
\end{align*}
Similarly, one sees that
\begin{align*}
\|w\tilde{\rho} F\|_{k} &\leq
 c_6^\frac{1}{k} \|w\tilde{\rho}\|_{\infty}^\frac{k-1}{k} \bigg( \int_{a_0}^{b_0} \zeta_2(s) \, ds \bigg)^\frac{1}{k},\\
 \bigg\|w\tilde{\rho} \sum_{i=1}^dD_iq_{ij}\bigg\|_{k}
 &\leq   c_5 \|w\tilde{\rho}\|_{\infty}^\frac{k-1}{k} \bigg( \int_{a_0}^{b_0} \zeta_2(s) \, ds \bigg)^\frac{1}{k},\\
 \|w\tilde\rho\vartheta^{\frac{k-2}{k}}\|_{\frac{k}{2}}  &\leq
 c_1^\frac{2}{k} \|w\tilde{\rho}\|_{\infty}^\frac{k-2}{k} \bigg( \int_{a_0}^{b_0} \zeta_1(s) \, ds \bigg)^\frac{2}{k},\\
 \|\tilde{\rho}\A_0w \|_{\frac{k}{2}} &\leq
 c_3 \|w\tilde{\rho}\|_{\infty}^\frac{k-2}{k} \bigg( \int_{a_0}^{b_0} \zeta_1(s) \, ds \bigg)^\frac{2}{k},\\
 \|\tilde{\rho}\partial_sw \|_{\frac{k}{2}}  &\leq
  c_4 \|w\tilde{\rho}\|_{\infty}^\frac{k-2}{k} \bigg( \int_{a_0}^{b_0} \zeta_1(s) \, ds \bigg)^\frac{2}{k},\\
  \|\tilde{\rho}(F\cdot\nabla w) \|_{\frac{k}{2}} &\leq
  \eta^{-1}c_2c_6^\frac{1}{k} \|w\tilde{\rho}\|_{\infty}^\frac{k-2}{k} \bigg( \int_{a_0}^{b_0} \zeta_2(s) \, ds \bigg)^\frac{2}{k}\,.
\end{align*}
Moreover, we have the estimate
\begin{align*}
\|w\tilde{\rho}\|_{\infty, 2} &\leq \|w\tilde{\rho}\|_{\infty}^\half \sup_{s \in (a_0,b_0)} \bigg(\int_{\CR^d}
\tilde\rho(s,x)w(s,x)\, dx\bigg)^\half\\
&\leq c_1^\half\|w\tilde{\rho}\|_{\infty}^\half \bigg (\sup_{s \in (a_0,b_0)} \zeta_1(s)\bigg )^{\frac{1}{2}}\,.
\end{align*}

Writing $X = \|w\tilde{\rho}\|_{\infty}^{\frac{1}{2k}}$ and inserting the above bounds back  into \eqref{eq.inftyest}, we obtain
\begin{equation}\label{eq.z}
X^k \leq \alpha + \beta X^{k-2} + \gamma X^{k-4}\,,
\end{equation}
where
\begin{eqnarray*}
\alpha = Cc_1^\half \bigg (\sup_{s \in (a_0,b_0)}\zeta_1(s)\bigg )^\half,\qquad\;\,
\beta = C(c_2+ c_5d + c_6^\frac{1}{k} ) \bigg( \int_{a_0}^{b_0} \zeta_2(s)\, ds \bigg)^\frac{1}{k}
\end{eqnarray*}
and
\begin{eqnarray*}
\gamma = C\bigg( \frac{kc_1^\frac{2}{k}}{b_0-b} +c_3+ c_4 \bigg) \bigg( \int_{a_0}^{b_0} \zeta_1(s)\, ds \bigg)^\frac{2}{k}
+C\eta^{-1}c_2c_6^\frac{1}{k} \bigg( \int_{a_0}^{b_0} \zeta_2(s)\, ds \bigg)^\frac{2}{k}\, .
\end{eqnarray*}
Young's inequality yields
\begin{eqnarray*}
X^{k-4} \leq \frac{k-4}{k-2}X^{k-2} + \frac{2}{k-2}\,,
\end{eqnarray*}
which, combined with \eqref{eq.z}, implies
\begin{equation}\label{eq.superz}
X^k \leq \tilde{\alpha} +\tilde{\beta}X^{k-2}
\end{equation}
where
\begin{eqnarray*}
\tilde{\alpha} = \alpha + \frac{2\gamma}{k-2}\quad\mbox{and}\quad \tilde{\beta} = \beta +\frac{k-4}{k-2}\gamma\, .
\end{eqnarray*}
We claim that \eqref{eq.superz} implies that $X \leq \tilde{\beta}^\half + \tilde{\alpha}^\frac{1}{k}$. To see this, we observe that
 $p(X) :=
X^{k-2}(X^2-\tilde{\beta})$ is a polynomial with zeros $0$ and $\pm\tilde{\beta}^\half$ which is increasing on
$(\tilde{\beta}^\half, \infty)$.
Moreover,
\begin{eqnarray*}
p(\tilde{\beta}^\half +\tilde{\alpha}^\frac{1}{k}) = \big(\tilde{\beta}^\half + \tilde{\alpha}^\frac{1}{k}\big)^{k-2}\big((\tilde{\beta}^\half +\tilde{\alpha}^\frac{1}{k})^2 - \tilde{\beta}\big) \geq \tilde{\alpha}^\frac{k-2}{k}(
\tilde{\beta} + \tilde{\alpha}^\frac{2}{k} -\tilde{\beta})
= \tilde{\alpha}\,,
\end{eqnarray*}
whence $X \geq \tilde{\beta}^\half +\tilde{\alpha}^\frac{1}{k}$ implies $p(X) \geq \tilde{\alpha}$. As \eqref{eq.superz} implies $p(X) \leq \tilde{\alpha}$, it follows that $X \leq \tilde{\beta}^\half + \tilde{\alpha}^\frac{1}{k}$ as claimed.

By the definition of $X, \tilde{\beta}$ and $\tilde{\alpha}$, we have proved that
\begin{eqnarray*}
\|w\tilde\rho\|_{\infty} \leq 2^{2k-1} (\tilde{\beta}^{k} + \tilde{\alpha}^2) \leq 2^{2k-1} \big( 2^{k-1}(\beta^{k} + \gamma^{k}) + 2(\alpha^2 + \gamma^2)\big)\, .
\end{eqnarray*}
Taking into account the definitions of $\alpha, \beta$ and $\gamma$, we get \eqref{eq.final} for a certain constant $C$, possibly different from $C$ above but only depending on $d, k$ and $\eta$, and with $c_2^{2k}$ and $c_2^4$ being replaced by zero. In particular, $C$ does not depend on $w$.\medskip

We now remove the additional assumption that $w$ is bounded by using Hypothesis \ref{hyp2}(1). Given $w$ as in the statement of the theorem, we
define, for $\eps>0$, $w_\eps := \frac{w}{1+\eps w}$. Then $w_\eps$ is bounded. Moreover, $\partial_s w_\eps
= (1+ \eps w)^{-2}\partial_sw$ and $\nabla w_\eps
= (1+ \eps w)^{-2}\nabla w$ are bounded since $w^{-2}\partial_sw$ and $w^{-2}\nabla w$ are.
This shows that Hypothesis \ref{hyp2}(1) holds for $w_\eps$.
We claim that $w_\eps$ also satisfies Hypothesis
\ref{hyp2}(2) with the same constants $c_j$, $j\neq 3$, and with $c_3$ being replaced by $\tilde{c}_3:=c_3+2\eta^{-1}c_2^2$.
We limit ourselves to checking condition (iii) in Hypothesis \ref{hyp2}(2) since the other conditions are straightforward. For this purpose, we observe that
\begin{align*}
\A_0w_{\eps}=(1+\eps w)^{-2}\A_0w-2\eps(1+\eps w)^{-3}\langle Q\nabla w,\nabla w\rangle\,,
\end{align*}
so that
\begin{align*}
|\A_0w_{\eps}| &\le (1+\eps w)^{-2}|\A_0w|+2\eps(1+\eps w)^{-3}|Q^{1/2}\nabla w|^2\\
&\le (1+\eps w)^{-2}|\A_0w|+2\eta^{-1}\eps(1+\eps w)^{-3}|Q\nabla w|^2\,.
\end{align*}
Using the conditions (ii) and (iii) for $w$ we get
\begin{align*}
|\A_0w_{\eps}|& \le c_3(1+\eps w)^{-2}w^{\frac{k-2}{k}}W_1^{\frac{2}{k}}+
2\eps\eta^{-1}(1+\eps w)^{-3}c_2^2w^{\frac{2k-2}{k}}W_1^{\frac{2}{k}}\\
&\le (c_3+2\eta^{-1}c_2^2)w_{\eps}^{\frac{k-2}{k}}W_1^{\frac{2}{k}}\,,
\end{align*}
which proves condition (iii).

The first part of the proof thus yields that \eqref{eq.final} holds if we replace $w$ with $w_\eps$ and $c_3$ with $\tilde{c}_3$. Note that the constant $C$ does not
depend on $\eps$. We may let $\eps \to 0$ and obtain \eqref{eq.final} for the original $w$.
\end{proof}

We next turn to the case of possibly unbounded diffusion coefficients. We extend Theorem \ref{main-bounded} to this general situation by using the approximating operators from Section \ref{subsect.approx}. We first make the following observation.

\begin{lem}\label{l.localize}
Assume Hypothesis \ref{hyp2} and let the coefficients $Q_n$ be given via equation \eqref{eq.qn}, where we pick $W=W_1$.
We define the operators $\A_n$ by \eqref{eq.an}. Then  the operators $\A_n$ satisfy Hypothesis \ref{hyp2}(1), (2)
with coefficients $\tilde{c}_1, \ldots, \tilde{c}_6$, given by $\tilde{c}_i = c_i$ for $i =1,4,6$ and
\begin{equation}\label{eq.ctilde}
\tilde{c}_2 :=2c_2, \quad \tilde{c}_3 := c_3 + \eta c_7, \quad \mbox{and}\quad \tilde{c}_5 := c_5 +4c_1c_8\, .
\end{equation}
\end{lem}

\begin{proof}
The estimates involving $c_1,c_4$ and $c_6$ obviously hold when $\A$ is replaced with $\A_n$. Let us now note that
\begin{eqnarray*}
|\nabla w| = |Q^{-1}Q \nabla w| \leq \eta^{-1} c_2w^\frac{k-1}{k}W_1^\frac{1}{k}\,.
\end{eqnarray*}
Similarly, we see that $|\nabla W_1| \leq \eta^{-1}c_8w^\frac{k-1}{k}W_2^\frac{1}{k}$.
It follows that
\begin{eqnarray*}
|Q_n\nabla w| =|\varphi_nQ\nabla w + (1-\varphi_n)\eta \nabla w| \leq |Q\nabla w| + \eta |\nabla w| \leq
2c_2w^\frac{k-1}{k}W_1^\frac{1}{k}\, .
\end{eqnarray*}
Moreover,
\begin{eqnarray*}
|\A_0^nw| \leq |\A_0w| + \eta |\Delta w| \leq (c_3+\eta c_7) w^\frac{k-2}{k}W_1^\frac{2}{k}\, .
\end{eqnarray*}
Finally, we have
\begin{eqnarray*}
\sum_{i=1}^dD_iq_{ij}^n = \varphi_n \sum_{i=1}^dD_iq_{ij} + \frac{1}{n} \varphi'\big(W_1/n\big )[(Q\nabla W_1)_j-\eta D_jW_1]\,.
\end{eqnarray*}
Noting that
\begin{eqnarray*}
\begin{aligned}
\bigg|\frac{w}{n}\varphi'\big(W_1/n\big)[(Q\nabla W_1)_j-\eta D_jW_1]\bigg|
 & \leq c_1\bigg|\frac{W_1}{n}\varphi' \big(W_1/n\big)\bigg| \big(|Q\nabla W_1|+\eta|\nabla W_1|\big)\\
 & \leq 2c_1 \big(|Q\nabla W_1|+\eta|\nabla W_1|\big)\,,
 \end{aligned}
\end{eqnarray*}
since $|x\varphi'(x)| \leq 2$ for any $x\in\CR$, it follows that
\begin{eqnarray*}
\bigg|w \sum_{i=1}^d D_iq_{ij}^n\bigg| \leq \bigg|w \sum_{i=1}^d D_iq_{ij}\bigg|+
2c_1 \big(|Q\nabla W_1|+\eta|\nabla W_1|\big)
\leq (c_5+4c_1c_8)w^\frac{k-1}{k}W_2^\frac{1}{k}\,,
\end{eqnarray*}
finishing the proof.
\end{proof}

We can now prove the main result of this section.

\begin{thm}\label{t.pointwise}
Let Hypothesis \ref{hyp2} be satisfied. Then there exists a positive constant $C$, depending only on $\eta$ and $d$, such that
\begin{align}
 w\rho
 \leq  C \bigg [  c_1 & \sup_{s \in (a_0,b_0)} \zeta_1(s)
+ (c_1^kc_8^k+c_2^{k} + c_5^{k}+ c_6 ) \int_{a_0}^{b_0} \zeta_2(s)\, ds\notag\\
& +\bigg( \frac{k^kc_1^2}{(b_0-b)^{k}}+c_2^{2k}+c_3^{k} + c_4^{k}+c_7^k\bigg) \bigg( \int_{a_0}^{b_0} \zeta_1(s)\, ds \bigg)^2\notag\\
&+c_2^{k}c_6\bigg( \int_{a_0}^{b_0} \zeta_2(s)\, ds \bigg)^2
 +c_2^2c_6^\frac{2}{k}\bigg( \int_{a_0}^{b_0} \zeta_2(s)\, ds \bigg)^\frac{4}{k}
\notag\\
& +
\bigg( \frac{k^2c_1^\frac{4}{k}}{(b_0-b)^{2}}+c_2^4 +c_3^2+c_4^{2}+c_7^2\bigg) \bigg( \int_{a_0}^{b_0} \zeta_1(s)\, ds \bigg)^\frac{4}{k} \bigg]
\label{eq-final-thm}
\end{align}
in $(a,b)\times\CR^d$.
\end{thm}

\begin{proof}
To approximate the coefficients $\A$, we use part (3) in Hypothesis \ref{hyp2}.
More precisely, for $n\in\CN$, we consider the operator $\A_n(t)$ defined in \eqref{eq.an}. By virtue of Lemma \ref{l.localize}
each operator $\A_n(t)$ has bounded diffusion coefficients and it satisfies Hypothesis \ref{hyp2}(1) and (2) with the same constants $c_j$ except
for $c_2$, $c_3$ and $c_5$ which are now replaced by $2c_2$, $c_3+\eta c_7$ and $c_5+4c_1c_8$, respectively. Take $\mu_s^n=G_n(t,s)^\ast \mu_t,\,s\in [0,t]$ for a fixed $t\in (0,1]$, where $G_n(t,s)$ is the evolution family corresponding to $\A_n(t)$. Combining \eqref{eq.final}
with \eqref{eq.ctilde} it follows that estimate \eqref{eq-final-thm} holds true with $\rho$ replaced with $\rho_n$ and $\zeta_i$ replaced by $\zeta_{i,n}$ defined by $\zeta_{i,n}(s)=\int_{\CR^d}W_i(s)\,d\mu_s^n$. Moreover, the constant $C$ is independent of $n$.

To finish the proof, we let $n$ tend to $\infty$.
To prove the convergence of the terms $\sup_{s\in (a_0,b_0)}\zeta_{1,n}(s)$ and $\int_{a_0}^{b_0}\zeta_{i,n}(s)ds$ to
$\sup_{s\in (a_0,b_0)}\zeta_1(s)$ and $\int_{a_0}^{b_0}\zeta_i(s)ds$, respectively, it suffices to prove that
$\zeta_{i,n}$ tends to $\zeta_i$ uniformly in $(a_0,b_0)$ as $n\to \infty$.

To that end, first note that $w\rho_n$ converges locally uniformly to
$w\rho$ by Proposition \ref{p.conv}. Using the estimate $W_i\le c_0V^{1-\sigma}$ and H\"older's inequality, we can estimate
\begin{align}
 |\zeta_{i,n}(s)-\zeta_i(s)|
& \le  \int_{B_R}W_i(s)|\rho_n(s)-\rho(s)|\,dx\notag\\
&\quad +\int_{\CR^d\setminus B_R}W_i(s)\,d\mu_s^n+\int_{\CR^d\setminus B_R}W_i(s)\,d\mu_s^n\notag\\
&\le \|W_i\|_{L^{\infty}((a_0,b_0)\times B_R)}\|\rho_n-\rho\|_{L^{\infty}((a_0,b_0)\times B_R)}\omega_d R^d\notag\\
&\quad+c_0\left (\int_{\CR^d\setminus B_R}V\,d\mu_s^n\right )^{1-\sigma}(\mu_s^n(\CR^d\setminus B_R))^{\sigma}\notag\\
&\quad +c_0\left (\int_{\CR^d\setminus B_R}V\,d\mu_s\right )^{1-\sigma}(\mu_s(\CR^d\setminus B_R))^{\sigma}\,,
\label{ultimaa-bis}
\end{align}
for any $s\in (a_0,b_0)$, $n\in\CN$ and $i=1,2$, where $\omega_d$ denotes the measure of the ball $B_1$. Using equation\eqref{eq.vest}, which holds true also with the measure $\mu_s$ being replaced with $\mu_s^n$ since $V$ is a Lyapunov function also for the operator $\mathscr A_n$, we can easily deduce that the family of measures $\{\mu_s^n: s\in (a_0,b_0),\, n\in\CN\}$
and $\{\mu_s: s\in (a_0,b_0),\, n\in\CN\}$ are tight. Therefore,
\begin{equation}
\lim_{R\to \infty}\sup_{s\in (a_0,b_0), n\in\CN}\mu_s^n(\CR^d\setminus B_R)=
\lim_{R\to \infty}\sup_{s\in (a_0,b_0)}\mu_s(\CR^d\setminus B_R)=0\,.
\label{ultimaa-ter}
\end{equation}
Further, estimate \eqref{eq.vest} shows also that
\begin{eqnarray*}
\int_{\CR^d\setminus B_R}V\,d\mu_s^n+\int_{\CR^d\setminus B_R}V\,d\mu_s\le 2\int_{\CR^d}V\,d\mu_t+2M(t-a_0)\,,
\end{eqnarray*}
for any $s\in (a_0,b_0)$.
Replacing this estimate back into \eqref{ultimaa-bis} we see that $\zeta_{i,n}$ converges to $\zeta_i$ uniformly in
$(a_0,b_0)$ as $n\to\infty$. Indeed, given $\eps>0$, we can first pick $R$ large enough, so that the last two terms in
\eqref{ultimaa-bis} are less than $\eps$ and then $n$ so large, that the first term is also less than $\eps$.
\end{proof}


\section{Pointwise estimates of the transition kernels}

We now make more specific assumptions on the structure of the operators $\A(t)$ in our parabolic equation and use Theorem \ref{t.pointwise}
to obtain estimates for the associated transition kernels. We make the following assumptions.

\begin{hyp}
\label{hyp3}
Assume that the coefficients $q_{ij}$ and $F_j$ ($i,j=1, \ldots, d$) satisfy parts (1) and (2) of Hypothesis \ref{hyp1}. Moreover, assume that
\begin{enumerate}
\item
there exist positive constants $\Lambda \geq 0$ and $m\geq 0$ such that
\begin{align*}
|Q(t,x)x| & \leq \Lambda (1+|x|^m)|x|\,,\\
\langle Q(t,x)\xi,\xi\rangle &\leq \Lambda (1+|x|^m)|\xi|^2\,,\\
|D_iq_{ij}(t,x)| & \leq \Lambda (1+|x|^m)\,,
\end{align*}
for all $(t,x)\in [0,1]\times\CR^d$, $\xi\in\CR^d$ and $i,j=1,\ldots,d$;
\item
there exist a nonnegative function $b:[0,1]\times\CR^d\to\CR$ and positive constants $p>\max\{m-1,1\}$, $\kappa$ and $K\ge 1$ such that
\begin{align}
|F(t,x)| &\le \Lambda |x|^p,&\quad \mbox{for all}\,\, (t,x)\in [0,1]\times\CR^d\,,\label{cond-B1}\\
\langle F(t,x),x\rangle & \le -b(t,x)|x|^{p+1},&\quad \mbox{for all}\,\, (t,x)\in [0,1]\times\CR^d\,,\label{cond-B2}\\
b(t,x) & \ge \kappa,&\quad \mbox{for all}\,\, (t,x)\in [0,1]\times (\CR^d\setminus B_K)\,.
\label{cond-B3}
\end{align}
\end{enumerate}
\end{hyp}

We show that assuming Hypothesis \ref{hyp3}, we are in the situation considered so far.

\begin{lem}\label{l.lyapunov}
Assume Hypothesis \ref{hyp3}, set $\beta:= p+1-m$ and let $\upsilon \in C^2(\CR^d)$ be such that
$\upsilon (x) = |x|^\beta$ for $|x|\geq 1$.
\begin{enumerate}
\item For $0<\delta < \kappa(\beta\Lambda)^{-1}$, part (3) of Hypothesis \ref{hyp1} is satisfied for the function $V$, defined by
$V(x) := \exp (\delta \upsilon (x))$.
\item Let $t \in (0,1]$ and $V$, $\delta$ be as in part (1). For $\eps < \delta$ and $\alpha> (p+1-m)/(p-1)$, the function $W: [0,t]\times \CR^d \to \CR$, defined by $W(s,x) := \exp (\eps (t-s)^\alpha \upsilon (x))$ is a Lyapunov function in the sense of Definition \ref{defn-2.6}.
\end{enumerate}
\end{lem}

\begin{proof}
(1) Straightforward computations show that for $|x|\geq 1$ and $i,j=1, \ldots, d$ we have
\begin{align*}
D_iV(x) & =\delta\beta x_i|x|^{\beta-2}V(x),\\
D_{ij}V(x) & =\delta\delta_{ij}\beta|x|^{\beta-2}V(x)
+ \delta\beta (\beta-2) x_ix_j|x|^{\beta-4}V(x) +\delta^2\beta^2x_ix_j|x|^{2(\beta-2)}V(x)\,.
\end{align*}
Thus, using Hypothesis \ref{hyp3}(1), \eqref{cond-B2} and \eqref{cond-B3} we get for $|x| \geq K$ and $t \in [0,1]$
\begin{align*}
\A(t)V(x)
& = \delta \beta |x|^{\beta-2 }V(x) \bigg [{\rm Tr}(Q(t,x))
+(\beta-2)\frac{\langle Q(t,x)x,x\rangle}{|x|^2}\\
&\qquad\qquad\qquad\qquad  +\delta\beta |x|^{\beta}\frac{\langle Q(t,x)x,x\rangle}{|x|^2}
+\frac{\langle F(t,x),x\rangle}{|x|^2}\bigg ]\\
& \le  \delta\beta |x|^{\beta-2}V(x)\Big [(d+(\beta-2)_+)\Lambda (1+|x|^m)\\
& \qquad\qquad\qquad\qquad
+ \delta\beta\Lambda |x|^\beta  (1+|x|^m)
-\kappa|x|^{p+1}\Big]\,.
\end{align*}
Noting that by assumption $\beta +m = p+1 > m$ and $\delta\beta \Lambda < \kappa$, it follows that
\begin{eqnarray*}
\limsup_{|x|\to\infty}\Big [(d+(\beta-2)_+)\Lambda (1+|x|^m)
+ \delta\beta |x|^\beta \Lambda (1+|x|^m)
-\kappa|x|^{p+1}\big] = -\infty\,.
\end{eqnarray*}
Thus, picking $K_1$ large enough, we have $\A(t)V(x)\le 0$  for all $(t,x) \in [0,1]\times (\CR^d\setminus B_{K_1})$.
Since $\A(\cdot)V$ is bounded in $[0,1]\times B_{K_1}$ we have $\A(t)V(x)\le M$ for all $(t,x)\in [0,1]\times\CR^d$ and some positive constant $M$.

The same arguments show that
\begin{align*}
\eta\Delta V(x)+F(t,x)\cdot \nabla V(x)\le  \delta\beta|x|^{\beta-2}\Big [\eta(d+(\beta-2)_+)
+ \delta\beta\eta |x|^{\beta}-\kappa|x|^{p+1}\Big ]V(x)\,,
\end{align*}
for any $|x|\ge 1$. From Hypothesis \ref{hyp3}(1), it follows that $\eta \le \Lambda$. So, by the
assumption $\delta \beta \Lambda <\kappa$, we have $\delta \beta \eta <\kappa$ and hence
the right-hand side of the above equation obviously converges to $-\infty$ as $|x|\to \infty$. We see that also
$\eta\Delta V(x)+F(t,x)\cdot\nabla V(x)$ is bounded on $[0,1]\times\CR^d$. This finishes the proof of part (1).\medskip

(2) Let us first note that since $\eps < \delta$, we find $W(s,x) \leq (V(x))^{\varepsilon/\delta}<V(x)$ for all $s\in [0,t], x \in \CR^d$. It is immediate
from the definition of $W$ that $W(s,x) \to \infty$ as $|x|\to \infty$, uniformly for $s$ in compact subsets of $[0,t)$.
It thus remains to prove part (3) in Definition \ref{defn-2.6}.

With similar computations and estimates as in part (1), we find for $|x| \geq K$
\begin{align}
&\phantom{=}  \partial_sW(s,x) - \A (s)W(s,x)\notag\\
& = - \eps\alpha (t-s)^{\alpha-1} |x|^{\beta} W(s,x) - \eps\beta (t-s)^\alpha |x|^{\beta-2} W(s,x) \bigg[{\rm Tr}(Q(t,x))
\notag\\
& \qquad  +(\beta-2)\frac{\langle Q(t,x)x,x\rangle}{|x|^2}+\eps\beta(t-s)^\alpha |x|^{\beta}\frac{\langle Q(t,x)x,x\rangle}{|x|^2}
+\frac{\langle F(t,x),x\rangle}{|x|^2}\bigg]\notag\\
& \geq - \eps\alpha (t-s)^{\alpha-1} |x|^{\beta} W(s,x)-\eps\beta (t-s)^\alpha |x|^{\beta-2} W(s,x) \times\label{eq.lyapest1}\\
&\qquad \times \Big [(d+(\beta-2)_+)\Lambda (1+|x|^m)
+ \eps\beta |x|^\beta \Lambda (1+|x|^m)
-\kappa|x|^{p+1}\Big]\notag\\
& = - \eps\alpha (t-s)^{\alpha-1} |x|^{\beta} W(s,x)-\eps\beta (t-s)^\alpha |x|^{\beta-2} W(s,x) \times\notag\\
&\qquad \times \Big[(d+(\beta-2)_+)\Lambda (1+|x|^m)
+ \delta\beta |x|^\beta \Lambda (1+|x|^m) -\kappa|x|^{p+1}\Big]\notag\\
& \qquad + \eps\beta^2(\delta-\eps)\Lambda (t-s)^\alpha |x|^{2\beta-2}(1+|x|^m)W(s,x)\notag\\
& \geq \eps (t-s)^{\alpha-1} |x|^\beta \Big( (\delta-\eps)(t-s)\beta^2 \Lambda |x|^{\beta+m-2} -\alpha\Big)W(s,x)\label{eq.lyapest2}\\
& \qquad -\eps\beta (t-s)^\alpha |x|^{\beta-2} W(s,x)\Big [(d+(\beta-2)_+)\Lambda (1+|x|^m)\notag\\
&\qquad\qquad\qquad\qquad\qquad\qquad\qquad\;\,
+ \delta\beta\Lambda |x|^\beta  (1+|x|^m) -\kappa|x|^{p+1}\Big]\,.\notag
\end{align}

We now estimate this further, distinguishing two cases. To that end, we set $C:= \max\{((\delta-\eps)\beta^2\Lambda/\alpha)^{-1/(\beta+m-2)}, K_1\}$, where $K_1$ is as in part (1), so that the second summand in \eqref{eq.lyapest2} is positive for $|x| \geq K_1$.\smallskip

\emph{Case 1:} $|x| \geq C(t-s)^{-1/(\beta+m-2)}$.

In this case, $(\delta-\eps)\beta^2\Lambda (t-s)|x|^{\beta+m-2}\geq \alpha$ and thus the first summand in \eqref{eq.lyapest2} is positive. Since $|x] \geq K_1$, also the second one is positive, so that overall $(\partial_s-\A (s))W(s,x) \geq 0$
in this case.\smallskip

\emph{Case 2:} $K\le |x| \leq C(t-s)^{-1/(\beta+m-2)}$.

In this case we drop the term involving $\kappa$ in \eqref{eq.lyapest1}, and we rewrite the estimate for $W^{-1}(\partial_sW-\A W)$ expanding all the terms as follows:
\begin{align}
&\frac{\partial_sW(s,x)-\A(s)W(s,x)}{W(s,x)}\notag\\
\ge &-\varepsilon\alpha(t-s)^{\alpha-1}|x|^{\beta}
-\varepsilon\beta(t-s)^{\alpha}(d+(\beta-2)_+)\Lambda |x|^{\beta-2+m}\notag\\
&-\varepsilon^2\beta^2(t-s)^{\alpha}\Lambda |x|^{2\beta-2+m}
-\varepsilon^2\beta^2(t-s)^{\alpha}\Lambda |x|^{2\beta-2}\notag\\
&-\varepsilon\beta(t-s)^{\alpha}(d+(\beta-2)_+)\Lambda |x|^{\beta-2}\,.
\label{estima-dep-beta}
\end{align}
The powers of $|x|$ in the first three terms in the right-hand side of \eqref{estima-dep-beta} are all positive. Hence, these terms can be estimated from below, replacing $|x|$ by $C(t-s)^{-1/(\beta+m-2)}$. On the contrary the sign of the powers of $|x|$ in the last two terms of
\eqref{estima-dep-beta} depends on the value of $\beta$. If $\beta\ge 2$ the powers of $|x|$ are nonnegative so that we can estimate the last two terms in \eqref{estima-dep-beta} as we did for the first three terms.
We conclude that
\begin{align*}
&\phantom{=}  \partial_sW(s,x) - \A (s)W(s,x)\\
& \geq -  \eps\alpha C^\beta (t-s)^{\alpha-1- \frac{\beta}{\beta+m-2}}W(s,x) -
\eps\beta C^{\beta-2}(t-s)^{\alpha-\frac{\beta-2}{\beta+m-2}}W(s,x)\times\\
&\qquad \times\Lambda\bigg (d+(\beta-2)_+ + \eps\beta C^\beta (t-s)^{-\frac{\beta}{\beta+m-2}}\bigg )\bigg (1+ C^m(t-s)^{-\frac{m}{\beta+m-2}}\bigg )\\
& \geq -  \eps\alpha C^\beta (t-s)^{\alpha-1- \frac{\beta}{\beta+m-2}}W(s,x) -\eps\beta C^{\beta-2}(t-s)^{\alpha- \frac{\beta-2}{\beta+m-2}}W(s,x)\times\\
&\qquad \times(t-s)^{-\frac{\beta+m}{\beta+m-2}}\big(( d+(\beta-2)_++\eps\beta C^\beta)\Lambda(1+C^m)\big)\\
& = -C_1 (t-s)^{\alpha-\frac{2\beta+m-2}{\beta-2 +m}} W(s,x)\,,
\end{align*}
for a suitable constant $C_1$.
On the other hand, if $\beta\in [1,2)$, the power of $|x|$ in the last term is negative. Hence, we estimate
$|x|^{\beta-2}\le K^{\beta-2}$ and again we conclude that
\begin{align}
\partial_sW(s,x) - \A (s)W(s,x)
\ge -C_2(t-s)^{\alpha-\frac{2\beta+m-2}{\beta-2+m}}W(s,x)\,,
\label{estim-C2}
\end{align}
for a positive constant $C_2$. If $\beta<1$ the powers of $|x|$ in the last two terms of \eqref{estima-dep-beta} are negative.
 Arguing as above, we can still prove estimate \eqref{estim-C2} possibly with a different constant $C_2$.
\smallskip

Combining the two cases, we see that $\partial_s W(s,x) - \A(s)W(s,x) \geq -\tilde{h}(s)W(s,x)$ for all $s \in [0,t)$
and $|x| \geq K$, where $\tilde h(s)=C(t-s)^{\alpha-\frac{2\beta +m-2}{\beta-2+m}}$ which belongs to $L^1((0,t))$ since $\alpha>\beta/(\beta-2+m)$.
To cover also the case $|x| \leq K$, observe that, by continuity, the function $(s,x) \mapsto |W(s,x)^{-1}(
\partial_s W(s,x) -\A(s)W(s,x))|$ is bounded on $[0,t]\times B_K$, say by $\tilde M$. Thus, if we define $h(s) :=
\max\{\tilde M/\tilde{h}(s), 1\} \tilde{h}(s)$, then $\partial_sW(s,x) -\A(s)W(s,x) \geq -h(s)W(s,x)$ for all
$(s,x) \in (0,t)\times\CR^d$.

The analogous estimate for $\eta \Delta + F\cdot \nabla$ follows by observing that
\begin{eqnarray*}
\eta\Delta W(s,x) + F(s,x)\cdot \nabla W(s,x) \leq \A(s)W(s,x),\qquad\;\, s\in (0,t),\;\,|x|\ge 1\,.
\end{eqnarray*}
This finishes the proof of part (2).
\end{proof}

Assuming Hypothesis \ref{hyp3}, we can now prove the following kernel estimates.

\begin{thm}
\label{thm.5.2}
Assume Hypothesis \ref{hyp3}, fix $\alpha > (p+1-m)/(p-1)$ and $k>d+2$. Then for every $0<\delta_0<
\kappa(\Lambda(p+1-m))^{-1}$, there exists a positive constant $\tilde C$ such that
\begin{equation}\label{estim-ex}
p_{t,s}(x,y)\leq \tilde C\left [(t-s)^{1-\frac{\alpha k(p\vee m)}{p+1-m}}+(t-s)^{2-\frac{\alpha k(p+(m-1)_+)}{p+1-m}}\right ]
e^{-\delta_0(t-s)^{\alpha}|y|^{p+1-m}}\,,
\end{equation}
for any $t\in (0,1]$, any $s\in (0,t)$ and any $x,y\in\CR^d$.
\end{thm}

\begin{example}
A concrete example of operators to which Theorem \ref{thm.5.2} applies is given
by the operators $\A(t)$, defined on smooth functions $\varphi$ by
\begin{eqnarray*}
(\A(t)\varphi)(x)=(1+|x|^m){\rm Tr}(Q^0(t,x)D^2\varphi(x))-b(t,x)|x|^{p-1}\langle x,\nabla\varphi(x)\rangle\,,
\end{eqnarray*}
for any $t\in [0,1]$ and any $x\in\CR^d$. Here $p,m$ and $b$ are as in Hypothesis \ref{hyp3} and $Q^0$ is a matrix valued function which is positive definite (i.e.\ there exists some $\eta >0$ such that $\langle Q^0(t,x)\xi,\xi\rangle\ge\eta |\xi|^2$ for all $(t,x)\in [0,1]\times\CR^d$ and $\xi\in\CR^d$) and such that  the entries $q_{ij}^0$ of the matrix $Q^0$ are
bounded, locally H\"older continuous in $[0,1]\times\CR^d$, continuously differentiable with respect to the spatial variables with bounded derivatives.

In the particular case when $Q^0$ is the identity matrix and $b\equiv 1$ we recover the time independent operator $\A_1$,
defined by
\begin{eqnarray*}
\A_1\varphi (x) := \Delta - |x|^{p-1}\langle x, \nabla \varphi (x)\rangle \, .
\end{eqnarray*}
In this situation our estimate \eqref{estim-ex} agrees with the result stated in \cite[Example 3.3]{alr10}.

Another particular case is the operator $\A_2$ defined by
\begin{eqnarray*}
\A_2\varphi (x) := (1+|x|^m)\Delta - |x|^{p-1}\langle x, \nabla \varphi (x)\rangle \, ,
\end{eqnarray*}
where $m>0$ and $p\ge \max\{m-1,1\}$. Applying Theorem \ref{thm.5.2} we obtain the following estimate for the associated transition kernel
\begin{eqnarray*}
p_t(x,y)\le C\left [t^{1-\frac{\alpha k(p\vee m)}{p+1-m}}+t^{2-\frac{\alpha k(p+(m-1)_+)}{p+1-m}}\right ]
e^{-\delta_0t^{\alpha}|y|^{p+1-m}}\,,
\end{eqnarray*}
for $\alpha >(p+1-m)/(p-1),\,k>d+2$, any $t\in (0,1]$ and any $x,y\in\CR^d$.
Observe that such estimates are known for the invariant measure, that is, for the limit, as $t\to \infty$ of $p_t(x,y)$, see
\cite[Example 3.6]{ffmp09}.
\end{example}

\begin{proof}[Proof of Theorem \ref{thm.5.2}]

To prove the theorem, we apply Theorem \ref{t.pointwise} with the evolution system of measures $p_{t,s}(x_0, \cdot) :=
G(t,s)^*\delta_{x_0}$, where $x_0 \in \CR^d$ is fixed. The density $p_{t,s}(x_0,\cdot)$ is defined as usual.
We let $\alpha, \beta, \delta, \upsilon$ and $V$ be
as in Lemma \ref{l.lyapunov}. We then pick $0<\eps_0 < \eps_1<\eps_2 < \delta$ such that $k(\eps_1-\eps_0)
< (\eps_2-\eps_0)$ and define the functions $w, W_1, W_2$ by
\begin{eqnarray*}
w(s,x) := e^{\eps_0(t-s)^\alpha \upsilon (x)},\qquad W_1(s,x) := e^{\eps_1(t-s)^\alpha \upsilon (x)},\qquad
W_2(s,x) := e^{\eps_2(t-s)^\alpha \upsilon (x)}
\end{eqnarray*}
for $(s,x) \in [0,t]\times \CR^d$. Clearly, $w\leq W_1\leq W_2$ and it follows from Lemma \ref{l.lyapunov}
that $W_1$ and $W_2$ are time dependent Lyapunov functions. We now check that Hypothesis \ref{hyp2} is satisfied. For now, we  fix
$0<a_0<a<b<b_0<t$. In the end, when we apply Theorem \ref{t.pointwise}, we will make a particular choice.
\smallskip

We first note that $\partial_sw(s,x) = -\eps_0\alpha (t-s)^{\alpha -1}\upsilon (x)w(s,x)$. It now easily follows that $w^{-2}\partial_sw$ is bounded
on $[a_0,b_0]\times \CR^d$ as long as $b_0<t$. Similarly, one sees that $w^{-2}\nabla w$ is bounded.\smallskip

Let us now turn to the estimates required in parts (2) and (3) of Hypothesis \ref{hyp2}.

Since $\eps_0<\eps_1$, we can clearly choose $c_1 :=1$, independent of $(a_0,b_0)$. As for the constant $c_2$, we have to bound the expression
\begin{align*}
\frac{|Q(s,x)\nabla w(s,x)|}{w(s,x)^{1-1/k}W_1(s,x)^{1/k}} & =\eps_0\beta(t-s)^\alpha |x|^{\beta-2} |Q(s,x)x|\Big(\frac{w(s,x)}{W_1(s,x)}\Big)^\frac{1}{k}\\
& \leq \eps_0\beta\Lambda (t-s)^{\alpha}|x|^{\beta-1} (1+|x|^m)e^{-k^{-1}(\eps_1-\eps_0)(t-s)^\alpha |x|^\beta}\,,
\end{align*}
where the inequalities hold for $|x|\geq 1$. To bound the above we note that for $r,\gamma,y >0$, we have
\begin{eqnarray*}
y^\gamma e^{-ry^\beta} = r^{-\frac{y}{\beta}}(ry^\beta)^\frac{\gamma}{\beta}e^{-ry^\beta} \leq r^{-\frac{\gamma}{\beta}}\Big(
\frac{\gamma}{\beta}\Big)^\frac{\gamma}{\beta} e^{-\frac{\gamma}{\beta}} =: r^{-\frac{\gamma}{\beta}} C(\gamma,\beta)\, ,
\end{eqnarray*}
which follows from the fact that the maximum of the function $t \mapsto t^re^{-t}$ on $(0,\infty)$ is attained at the point $t=r$. Applying this estimate
in the case where $y = |x|$, $r= k^{-1}(\eps_1-\eps_0)(t-s)^\alpha$, $\beta=\beta$ and $\gamma = \beta-1+m$, we get
\begin{align*}
& \phantom{=} \frac{|Q(s,x)\nabla w(s,x)|}{w(s,x)^{1-1/k}W_1(s,x)^{1/k}}\\
&  \leq 2\eps_0\beta\Lambda (t-s)^\alpha \Big(\frac{\eps_1-\eps_0}{k}\Big)^{-\frac{\beta
-1+m}{\beta}}(t-s)^{-\alpha \frac{\beta-1+m}{\beta}} C(\beta-1+m, \beta)\\
& =: \bar{c}_2 (t-s)^{-\frac{\alpha(m-1)}{\beta}} \leq \bar{c}_2(t-b_0)^\frac{-\alpha(m-1)_+}{\beta}\, ,
\end{align*}
for all $|x|\geq1$. Since the quotient we have to estimate is bounded on $[0,t]\times B_1$, we can have the above estimate on all of
$[a_0,b_0]\times \CR^d$ at the cost of a possibly larger constant $\bar{c}_2$.\smallskip

The other constants $c_3,\ldots,c_8$ are obtained in a similar way. It turns out that they are always of the form $c_j = \bar{c}_j(t-b_0)^{-r_j}$
for a certain constant $\bar{c}_j$ (which may depend on the constants $d,p, m, k, \eps_0, \eps_1, \eps_2, \Lambda$ and the behaviour
of the function $\upsilon$ on $B_1$) and a certain exponent $r_j$. Of particular importance for us will be the exponents $r_j$. Therefore, to
simplify the presentation, we will drop constants from our notation and write $\lesssim$ to indicate an estimate involving a constant
depending only on the quantities just mentioned.

Concerning $c_3$, we find
\begin{align*}
&\frac{|\A_0w(s,x)|}{w(s,x)^{1-2/k}W_1(s,x)^{2/k}}\\
\lesssim &
\big [(t-s)^{\alpha}|x|^{\beta-2+m}+(t-s)^{2\alpha}|x|^{2\beta -2+m}\big ] e^{-\frac{2}{k}(\eps_1-\eps_0)(t-s)^\alpha|x|^\beta}\\
\lesssim & (t-s)^{2\alpha} (t-s)^{-\alpha\frac{2\beta-2+m}{\beta}} \leq (t-b_0)^{-\frac{\alpha (m-2)_+}{\beta}}\,,
\end{align*}
so that $r_3 = \alpha (m-2)_+/\beta$. The estimates
\begin{align*}
\frac{|\partial_sw(s,x)|}{w(s,x)^{1-2/k}W_1(s,x)^{2/k}} & \lesssim (t-s)^{\alpha-1}|x|^{\beta} e^{-\frac{2}{k}(\eps_1-\eps_0)(t-s)^\alpha|x|^\beta}\\
& \lesssim (t-s)^{\alpha -1} (t-s)^{-\alpha} \leq (t-b_0)^{-1}\,,
\end{align*}
\begin{align*}
\frac{|\sum_{j=1}^dD_jq_{ij}(s,x)|}{w(s,x)^{-1/k}W_2(s,x)^{1/k}} \lesssim |x|^{m} e^{-\frac{1}{k}(\eps_2-\eps_0)(t-s)^\alpha|x|^\beta}
\lesssim (t-s)^{-\frac{\alpha m}{\beta}} \leq  (t-b_0)^{-\frac{\alpha m}{\beta}}
\end{align*}
and
\begin{align*}
\frac{w(s,x)|F(s,x)|^k}{W_2(s,x)} \lesssim |x|^{kp} e^{-(\eps_2-\eps_0)(t-s)^\alpha|x|^\beta}\lesssim (t-s)^{-\frac{\alpha kp}{\beta}} \leq  (t-b_0)^{-\frac{\alpha kp}{\beta}}
\end{align*}
yield $r_4=1$, $r_5= \alpha m/\beta$ and $r_6 = \alpha kp/\beta$, respectively. Finally, repeating the computations for $c_3$
with $m=0$ yields $r_7 = 0$ and similar computations as for $c_2$ yield $r_8 = \alpha (m-1)_+/\beta$.\medskip

Now, given $s \in (0,t)$, we choose
$a_0 = \max\{s- (t-s)/2, s/2\}$, $b=s+(t-s)/3$ and $b_0 = s+ (t-s)/2$ so that $b_0-b=(t-s)/6$, $t-b_0= (t-s)/2$ and $b_0-a_0 \leq t-s$.
Let us also note that as a consequence
of Proposition \ref{p.lyapunov}
\begin{eqnarray*}
\zeta_i(r) := \int_{\CR^d} W_i(r)\, d\mu_r \leq  \exp\left (\int_r^th(\tau)\,d\tau\right )\int_{\CR^d}W_i(t)\,d\delta_{x_0} =
\exp\left (\int_r^th(\tau)\,d\tau\right )\,,
\end{eqnarray*}
which, recalling the special form of $h$ from the proof of Lemma \ref{l.lyapunov} is easily seen to be bounded by a constant $M$
only depending on $\alpha$ and $\beta$. In particular, $M$ does not depend on $(a_0,b_0)$ so that
\begin{eqnarray*}
\int_{a_0}^{b_0} \zeta_i(r)\, dr \leq M(b_0-a_0) = M(t-s)\, .
\end{eqnarray*}

We now apply Theorem \ref{t.pointwise}, keeping track only of powers of $(t-s)$ and absorbing all other constants into the constant $C$
which exists by virtue of that theorem. We call the modified constant $C_1$. We have
\begin{align*}
wp_{t,s}&(x_0,\cdot) \leq C_1 \Big[ 1+ (t-s)\Big((t-s)^{- \frac{\alpha k (m-1)_+}{\beta}} + (t-s)^{-\frac{\alpha km}{\beta}} +
(t-s)^{-\frac{\alpha kp}{\beta}}\Big)\\
& + (t-s)^2\Big((t-s)^{-\frac{2\alpha k(m-1)_+}{\beta}} + (t-s)^{-\frac{\alpha k ((m-1)_+ +p)}{\beta}} +
(t-s)^{-\frac{\alpha k(m-2)_+}{\beta}} +1\Big)\\
& + (t-s)^2\frac{1}{(t-s)^k} + (t-s)^{\frac{4}{k}}\frac{1}{(t-s)^2}\\
& + (t-s)^{\frac{4}{k}}\Big((t-s)^{-\frac{4\alpha (m-1)_+}{\beta}} + (t-s)^{-\frac{2\alpha ((m-1)_+ + p)}{\beta}} +
(t-s)^{-\frac{2\alpha (m-2)_+}{\beta}} +1\Big)\Big]\,.
\end{align*}
To further simplify this, we only consider terms with the highest negative exponent, dropping all other terms at the cost of a possibly
larger constant $C_1$. We may thus drop all terms $1$.
We may also drop the entire fourth line, starting with $(t-s)^\frac{4}{k}$, as every
exponent appearing there is the $2/k$-fold (and thus with a smaller negative exponent) of an exponent appearing two lines above.
Moreover, the leading term on the second line is $(t-s)^{2-\frac{\alpha k ((m-1)_+ +p)}{\beta}}$, as it can be easily seen, taking into account that $p>m-1$. This term controls also
the second term in brackets on the first line, since $\alpha/\beta>1/(p-1)$ and $p>1$. Hence, the leading term on the first line
is $(t-s)^{1-\frac{\alpha kp}{\beta}}$. This term can be used also to bound the two terms on the third line,
again since $\alpha/\beta>1/(p-1)$.
We thus obtain
\begin{eqnarray*}
wp_{t,s}(x_0,\cdot) \leq C_2( (t-s)^{1-\frac{\alpha k (p\vee m)}{\beta}} + (t-s)^{2- \frac{\alpha k ((m-1)_+ + p)}{\beta}})
\end{eqnarray*}
for a certain constant $C_2$. Recalling the definitions of $w$ and $\beta$,
this is \eqref{estim-ex} as $C_2$ does not depend on $x_0$.
\end{proof}

\bibliographystyle{plain}

\end{document}